\documentclass[11pt]{article}
\usepackage[colorlinks=true,linkcolor=blue,citecolor=Red]{hyperref}
\usepackage{breakcites}

\usepackage{mathtools}

\usepackage{subfigure}

\usepackage{algorithm}
\usepackage{algpseudocode}

\usepackage{latexsym}
\usepackage{graphics}
\usepackage{amsmath}
\usepackage[algo2e]{algorithm2e}
\usepackage{xspace}
\usepackage{amssymb}
\usepackage{psfrag}
\usepackage{epsfig}
\usepackage{amsthm}
\usepackage{url}
\usepackage{pst-all}
\usepackage{bbm}

\usepackage[title]{appendix}

\usepackage{color}

\definecolor{Red}{rgb}{1,0,0}
\definecolor{Blue}{rgb}{0,0,1}
\definecolor{Olive}{rgb}{0.41,0.55,0.13}
\definecolor{Yarok}{rgb}{0,0.5,0}
\definecolor{Green}{rgb}{0,1,0}
\definecolor{MGreen}{rgb}{0,0.8,0}
\definecolor{DGreen}{rgb}{0,0.55,0}
\definecolor{Yellow}{rgb}{1,1,0}
\definecolor{Cyan}{rgb}{0,1,1}
\definecolor{Magenta}{rgb}{1,0,1}
\definecolor{Orange}{rgb}{1,.5,0}
\definecolor{Violet}{rgb}{.5,0,.5}
\definecolor{Purple}{rgb}{.75,0,.25}
\definecolor{Brown}{rgb}{.75,.5,.25}
\definecolor{Grey}{rgb}{.5,.5,.5}

\newcommand{\bs}{\boldsymbol{\sigma}}

\usepackage{bbm}
\newcommand{\ind}{\mathbbm{1}}

\def\OPT{{\mathsf{H^*}}}

\newcommand{\R}{\mathbb{R}}

\newcommand{\N}{\mathbb{N}}
\newcommand{\ip}[2]{\langle{#1},{#2}\rangle} 
\renewcommand{\ip}[2]{\left\langle#1,#2\right\rangle}

\renewcommand{\R}{\mathbb{R}}

\newcommand{\distr}{\stackrel{d}{=}}
\newcommand{\A}{\mathcal{A}}
\newcommand{\bincube}{\mathcal{B}_n}

\newcommand{\cN}{{\bf \mathcal{N}}}

\newcommand{\ignore}[1]{\relax}

\newtheorem{theorem}{Theorem}[section]
\newtheorem{remark}[theorem]{Remark}
\newtheorem{lemma}[theorem]{Lemma}

\newtheorem{proposition}[theorem]{Proposition}

\newtheorem{definition}[theorem]{Definition}

\newtheorem{Corollary}[theorem]{Corollary}

\renewcommand{\ip}[2]{\left\langle#1,#2\right\rangle}


\newcounter{parentnumber}

\makeatletter
\def\BState{\State\hskip-\ALG@thistlm}
\makeatother

\definecolor{Red}{rgb}{1,0,0}
\definecolor{Blue}{rgb}{0,0,1}
\definecolor{Olive}{rgb}{0.41,0.55,0.13}
\definecolor{Green}{rgb}{0,1,0}
\definecolor{MGreen}{rgb}{0,0.8,0}
\definecolor{DGreen}{rgb}{0,0.55,0}
\definecolor{Yellow}{rgb}{1,1,0}
\definecolor{Cyan}{rgb}{0,1,1}
\definecolor{Magenta}{rgb}{1,0,1}
\definecolor{Orange}{rgb}{1,.5,0}
\definecolor{Violet}{rgb}{.5,0,.5}
\definecolor{Purple}{rgb}{.75,0,.25}
\definecolor{Brown}{rgb}{.75,.5,.25}
\definecolor{Grey}{rgb}{.5,.5,.5}
\definecolor{Pink}{rgb}{1,0,1}
\definecolor{DBrown}{rgb}{.5,.34,.16}
\definecolor{Black}{rgb}{0,0,0}

\DeclareMathOperator*{\argmax}{arg\,max}

\usepackage{float}

\usepackage{float}
\title{Shattering in the Ising Pure $p$-Spin Model
}
\author{{\sf David Gamarnik}\thanks{Sloan School of Management, Massachusetts Institute of Technology; e-mail: {\tt gamarnik@mit.edu}.} 
\and
{\sf Aukosh Jagannath}\thanks{Department of Statistics and Actuarial Science, Department of Applied Mathematics, University of Waterloo; e-mail: {\tt a.jagannath@uwaterloo.ca}.}
\and
{\sf Eren C. K{\i}z{\i}lda\u{g}}\thanks{Department of Statistics, Columbia University; e-mail: {\tt eck2170@columbia.edu}.}
}
\begin{document}
\maketitle
\begin{abstract}
We study the Ising pure $p$-spin model for large $p$. 
We investigate the landscape of the Hamiltonian of this model. We show that for any $\gamma>0$ and any large enough $p$, the model exhibits an intricate geometrical property known as the multi Overlap Gap Property above the energy value $\gamma\sqrt{2\ln 2}$. 
We then show that for any inverse temperature $\sqrt{\ln 2}<\beta<\sqrt{2\ln 2}$ and any large $p$, the model exhibits \emph{shattering}: w.h.p.\,as $n\to\infty$, there exists exponentially many well-separated clusters such that (a) each cluster has exponentially small Gibbs mass, and (b) the clusters collectively contain all but a vanishing fraction of Gibbs mass. Moreover, these clusters consist of configurations with energy near $\beta$. Range of temperatures for which shattering occurs is within the \emph{replica symmetric} region.
To the best of our knowledge, this is the first shattering result regarding the Ising $p$-spin models. Our proof is elementary, and in particular based on simple applications of the first and the second moment methods.
\end{abstract}

\section{Introduction}
We prove the existence of a shattering phase for the Ising $p$-spin model at the level of the Gibbs measure for large enough $p$. In particular, we show that for all $\sqrt{\ln 2}< \beta<\sqrt{2\ln 2}$, the Gibbs measure at inverse temperature $\beta$ is shattered for the $p$-spin model for $p$ larger than some absolute constant. This regime of (inverse) temperatures was shown  to be in the replica symmetric phase for the Ising $p$-spin models for large $p$ by Talagrand \cite{talagrand2000rigorous}; in particular, the value $\sqrt{2\ln 2}$ is the replica symmetry breaking transition for the formal $p\to\infty$ limit of these models, namely Derrida’s Random Energy Model~\cite{derrida1980random,derrida1981random}. Along the way, we prove results regarding the landscape geometry of these models; we show in particular that they satisfy a certain type of clustering, and a version of the Overlap Gap Property, which is known to imply algorithmic hardness. 

The study of the free energy landscape and the geometry of mean-field spin glass models has a rich history in the physics literature. The $p$-spin model we consider here was introduced by Derrida in \cite{derrida1980random} as a generalization of the Sherrington-Kirkpatrick model \cite{sherrington1975solvable} (the case $p=2$) that allows for $p$-body interactions. More formally the $p$-spin Hamiltonian is given by 
\begin{equation}\label{eq:p-spin-Hamiltonian}
    H_{n,p}(\bs) = n^{-\frac{p+1}{2}} \ip{\boldsymbol{J}}{\bs^{\otimes p}} = n^{-\frac{p+1}{2}} \sum_{1\le i_1,\dots,i_p\le n}J_{i_1,\dots,i_p}\bs_{i_1}\cdots \bs_{i_p},
\end{equation}
where $p\ge 2$ is a fixed integer and $\boldsymbol{J}=(J_{i_1,\dots,i_p}:1\le i_1,\dots,i_p\le n)\in(\R^n)^{\otimes p}$ is an order-$p$ tensor whose entries are i.i.d.\,standard normal, $J_{i_1,\dots,i_p}\distr \cN(0,1)$.\footnote{It is also common to use the normalization $n^{-(p-1)/2}$ in the literature. 
} We focus here on the case that the configuration space is the discrete hypercube, $\bs\in\Sigma_n\triangleq\{-1,1\}^n$. The setting where the configuration space is the hypersphere of radius $\sqrt{n}$ is also of great interest. The former is sometimes called the \emph{Ising spin} models and the latter is called the \emph{spherical} models. In this work, we focus exclusively on the Ising spin models. 

When studying spin systems, two quantities play an essential role, namely the free energy
\begin{equation}\label{eq:limit-FE}
    F(\beta)\triangleq \lim_{n\to\infty}\frac{\ln Z_\beta}{n},\quad\text{where}\quad Z_\beta = \sum_{\bs\in\Sigma_n}e^{\beta n H_{n,p}(\bs)},
\end{equation}
and the Gibbs measure
\begin{equation}\label{eq:Gibbs-measure}
\mu_\beta(A) = \frac{1}{Z_\beta}\sum_{\bs \in A} e^{\beta n H_{n,p}(\bs)},
\end{equation}
where $\beta>0$ is the \emph{inverse temperature}.
Computing the free energy of mean-field spin glass models and the corresponding landscape has a large body of literature in the physics community for which we have no hope here of providing a complete summary. We instead point the reader to  textbook introductions~\cite{talagrand2010mean,panchenko2013sherrington}.   

In the mathematics literature, our understanding of the low temperature, or replica symmetry breaking, phase is by now fairly complete. In particular, the free energy was first computed for $p$ even by Talagrand in \cite{talagrand2006parisi} and for general $p$ by Panchenko in~\cite{panchenko2014parisi}, following the important works of Guerra \cite{guerra2003broken} and Aizenman-Sims-Star \cite{aizenman2003extended}, respectively. Far more is known, such as the construction of asymptotic Gibbs measures \cite{arguin2008remark}, ultrametricity of the asymptotic Gibbs measure  \cite{panchenko2013parisi}, the TAP equations \cite{auffinger2019thouless}, and generalizations of the TAP free energy \cite{subag2018free,chen2023generalized}. The phase diagram has also received a tremendous amount of attention \cite{toninelli2002almeida,auffinger2015properties,jagannath2017some}. In particular, it is known \cite{jagannath2017some} that the so-called de Almeida-Thouless line is the correct phase boundary for the SK model (up to a compact set away from the critical external field) but not for the $p$-spin models. Indeed the phase diagram for the $p$-spin models is expected to be particularly rich and to exhibit many hallmarks of spin glass behaviour.

One of the central predictions regarding the phase diagram of mean-field spin glasses is the existence of the \emph{shattering phase} (see Definition~\ref{def:shattering} below for a precise definition). The shattering phase was introduced by Kirkpatrick and Thirumalai in their landmark work \cite{kirkpatrick1987p}, where they predicted that it appears in Ising $p$-spin models. This notion has since played an important role in our understanding of spin glasses and, in particular, its connection to algorithmic hardness. Indeed, the important related problem of shattering at zero temperature (i.e., shattering of the energy landscape) has now been shown for constraint satisfaction problems such as \cite{achlioptas2006solution,achlioptas2011solution,krzakala2007gibbs,achlioptas2008algorithmic,sly2016reconstruction}.

In the mathematics literature, the existence of the shattering phase in a spin glass model was first proved in \cite{arous2021shattering} for the spherical $p$-spin model via the TAP complexity approach building on the important work of Subag \cite{subag2017extremal} and Auffinger-Ben Arous-\v{C}erny \cite{auffinger2013random}.
For the Ising $p$-spin model, however, to our knowledge this question has remained open since the original work of Kirkpatrick and Thirumalai in 1987.  In this paper we demonstrate the existence of this phase in the Ising setting for large $p$.

The approach we take here is via an observation that goes back to the orignal work of Derrida on the Random Energy Model \cite{derrida1980random,derrida1981random}, namely that the $p$-spin model for large $p$ is well approximated by the Random Energy Model. Indeed, using this perspective, we are able to provide an elementary second moment method approach to prove the existence of the shattering phase by taking $p$ to be large enough. Along the way, we prove new results regarding concentration of sub-level sets of the Hamiltonian and clustering, control the ground state energy, and also demonstrate $m$-Overlap Gap Property.

It is important here to compare our work with the earlier work of Talagrand \cite{talagrand2000rigorous} (see also~\cite[Chapter~16]{talagrand2011advanced}). To describe his results, let
\begin{equation}\label{eq:RSB}
    \beta_p\triangleq \sup\left\{\beta:\limsup_{n\to\infty}\frac{\mathbb{E}[\ln Z_\beta]}{n} = \ln 2 + \frac{\beta^2}{2}\right\}.
\end{equation}
The value $\ln 2 + \beta^2/2$ corresponds to the \emph{annealed free energy}, $\lim_{n\to\infty}(\ln\mathbb{E}[Z_\beta])/n$. In particular, the model is replica symmetric for $\beta<\beta_p$. Talagrand's first main result~\cite[Theorem~1.1]{talagrand2000rigorous}  shows that $\beta_p$ is asymptotically $\sqrt{2 \ln 2}$: for any $p$, $(1-2^{-p})\sqrt{2\ln 2}\le \beta_p\le \sqrt{2\ln 2}$. His next main result~\cite[Theorem~1.4]{talagrand2000rigorous} (see also~\cite[Theorems~16.3.6 and 16.4.1]{talagrand2011advanced}) shows that for all sufficiently large $p$ and any $\beta>0$ the Gibbs measure decomposes into \emph{lumps} $(\mathcal{C}_\alpha)_{\alpha\ge 1}$ such that (a) the lumps collectively contain all but a vanishing fraction of Gibbs mass, and (b) the overlap between two configurations in the same lump is close to 1 whereas the overlap between configurations belonging to different lumps is near zero. While these lumps are well-separated, this however does not quite correspond to shattering. It is not clear whether the number of lumps is exponential or whether each lump is sub-dominant.

{\bf A Concurrent Work.} We end here by noting a very recent concurrent work by El Alaoui, Montanari and Sellke~\cite{ams2023}. They show that for $p$  sufficiently large, the spherical pure $p$-spin model in fact exhibits shattering for a range of temperatures within the replica symmetric regime. The notion of shattering they consider is similar to ours, they establish the presence of a shattering phase at the level of the Gibbs measure\footnote{The prior work~\cite{arous2021shattering} establishes that the `free energy landscape' is shattered and the TAP free energy formula. On the other hand, the result of~\cite{arous2021shattering} holds for all $p\ge 4$, whereas that of~\cite{ams2023} as well as ours hold for large enough $p$.}. 
Their argument is based on estimates concerning the (derivative of) Franz-Parisi potential; it in particular requires studying a Parisi measure. 

\subsection{Algorithmically Finding a Near Ground-State} A fundamental quantity regarding the $p$-spin model is the \emph{ground-state energy}:
\begin{equation}\label{eq:Ground-En}
    \OPT \triangleq \lim_{n\to\infty}\max_{\bs\in\Sigma_n}H_{n,p}(\bs)
\end{equation}
The ground-state value~\eqref{eq:Ground-En} can be recovered as the zero temperature limit of the Parisi formula: $\OPT = \lim_{\beta\to\infty}F(\beta)/\beta$. Equipped with the ground-state value, a natural algorithmic question is finding a near ground-state efficiently (i.e., in polynomial time). That is, given a $\boldsymbol{J}\in(\R^n)^{\otimes p}$ and an $\epsilon>0$, the algorithmic task is to efficiently find a $\bs_{\rm ALG}\in\Sigma_n$ such that
$H_{n,p}(\bs_{\rm ALG}) \ge (1-\epsilon)\OPT$ say w.h.p.\footnote{To be more precise, the algorithm $\A$ receives tensors $\boldsymbol{J}_n\in(\R^n)^{\otimes p}$ and outputs a sequence $\bs_n = \A(\boldsymbol{J}_n)\in\Sigma_n$.} For the SK model ($p=2$), Montanari~\cite{montanari2019FOCS} devised an Approximate Message Passing (AMP) type algorithm which, for any $\epsilon>0$, finds a $\bs_{\rm ALG}\in\Sigma_n$ such that $H_{n,p}(\bs_{\rm ALG})\ge (1-\epsilon)\OPT$ w.h.p. His algorithm is based on an unproven (though widely believed) assumption that the underlying model does not exhibit the \emph{Overlap Gap Property} (OGP). At a high level, the OGP asserts that a certain `cluster' of near ground-state configurations is `forbidden', i.e.\,they do not occur w.h.p. See Section~\ref{sec:ogp-background} for details. Montanari's algorithm was inspired by an algorithm of Subag~\cite{subag2021following} regarding the  spherical mixed $p$-spin model and was subsequently extended to Ising mixed $p$-spin models~\cite{el2021optimization,sellke2021optimizing}; these algorithms also find (w.h.p.) a $\bs_{\rm ALG}\in\Sigma_n$ with $\widetilde{H}_{n}(\bs_{\rm ALG})\ge (1-\epsilon)\OPT$ (where $\widetilde{H}_{n}$ is the Hamiltonian of underlying mixed $p$-spin model) for any $\epsilon>0$, provided that the underlying model does not exhibit the OGP. For $p\ge 4$ though, the pure $p$-spin model is known to exhibit the OGP~\cite[Theorem~3]{chen2019suboptimality}, which was shown to be a rigorous barrier for AMP type algorithms~\cite{gamarnikjagannath2021overlap} (conditionally on a conjecture that has since been proved by Sellke~\cite{sellke2021approximate}). More concretely,~\cite{gamarnikjagannath2021overlap} showed that for any even $p\ge 4$, there exists a value $\bar{\mu}$ such that the value $H_{n,p}(\bs_{\rm ALG})$ returned by any AMP type algorithm is strictly below $\OPT-\bar{\mu}$ w.h.p. Subsequent work extended this hardness result to low-degree polynomials~\cite{gamarnik2020low} and to Boolean circuits of low-depth~\cite{gamarnik2021circuit}. The lower bounds established in~\cite{gamarnikjagannath2021overlap,gamarnik2020low,gamarnik2021circuit} are however not tight: they do not quite match the best known algorithmic threshold. More recently, Huang and Sellke~\cite{huang2021tight,huang2023algorithmic} invented a very sophisticated version of the OGP dubbed as the \emph{branching OGP} and subsequently established tight hardness guarantees against Lipschitz algorithms for the $p$-spin model. 

As mentioned earlier, the classical OGP established in~\cite{chen2019suboptimality} fails to yield tight lower bounds, see~\cite{gamarnikjagannath2021overlap,gamarnik2020low,gamarnik2021circuit}. In order to circumvent this issue and obtain tight lower bounds, one needs to rely on the branching OGP~\cite{huang2021tight,huang2023algorithmic}, which is a very intricate constellation of near ground-state solutions consisting of an ultrametric tree. Additionally, proofs of these OGP results are quite involved and require sophisticated technical tools. One of our main results, Theorem~\ref{thm:m-ogp}, shows that in the case of large $p$, one can consider a much simpler structure instead and establish its absence using an elementary argument based on the \emph{first moment method}. This structure is known as the symmetric $m$-OGP (see Section~\ref{sec:ogp-background}) and Theorem~\ref{thm:m-ogp} shows the following: for any $m\in\N$ and any $\gamma>1/\sqrt{m}$, there exists a $P_m\in \N$ such that for any fixed $p\ge P_m$, the set of $\bs$ with $H_{n,p}(\bs)\ge\gamma\sqrt{2\ln 2}$ exhibits symmetric $m$-OGP w.h.p.\,as $n\to\infty$. Furthermore, the symmetric $m$-OGP is also a barrier for stable algorithms: for any such $\gamma>0$ and $p\ge P_m$, there do not exist a sufficiently stable algorithm finding (w.h.p.\,as $n\to\infty$) a $\bs_{\rm ALG}$ with $H_{n,p}(\bs_{\rm ALG}) \ge \gamma\sqrt{2\ln 2}$. See the paragraph on algorithmic lower bounds on Section~\ref{sec:m-ogp} for details. Note that for large $m$, the onset of this property approaches zero, therefore coinciding with the algorithmic threshold for the REM~\cite{addario2020algorithmic}, see the remark following Theorem~\ref{thm:m-ogp}. 
\subsection{Background on Overlap Gap Property and Algorithmic Barriers}\label{sec:ogp-background}
Finding a near ground-state for the $p$-spin model is an example of a random optimization problem with a \emph{statistical-to-computational gap} (\texttt{SCG}): the best known efficient algorithm performs strictly worse than the existential guarantee. In the context of $p$-spin models, this means no efficient algorithm that finds (w.h.p.) a solution $\bs$ with $H(\bs)$ arbitrarily close to $\OPT$ is known. Other random optimization problems with an \texttt{SCG} include optimization over random graphs~\cite{gamarnik2014limits,gamarnik2017,rahman2017local}, random constraint satisfaction problems (CSPs)~\cite{gamarnik2017performance,bresler2021algorithmic}, and perceptron models~\cite{gamarnik2022algorithms,gamarnik2023geometric}. While the standard complexity theory is often useless in the random setting\footnote{See~\cite{ajtai1996generating,boix2021average,GK-SK-AAP} for a few exceptions. In particular,~\cite{GK-SK-AAP} establishes the average-case hardness of the algorithmic problem of \emph{exactly} computing $Z_\beta$ in~\eqref{eq:limit-FE} under the standard assumption $P\ne \#P$.}, an active line of research proposed various frameworks for giving `rigorous evidence' of hardness. We do not review these frameworks here, and refer the interested reader to surveys~\cite{kunisky2022notes,gamarnik2021overlap,gamarnik2022disordered}. One such framework in fact emerges from the study of spin glass models and is based on the intricate geometry of the space of near-optimal solutions. 
\paragraph{Overlap Gap Property (OGP)} For certain random CSPs, the works~\cite{mezard2005clustering,achlioptas2006solution,achlioptas2008algorithmic} discovered an intriguing connection between the solution space geometry and algorithmic hardness (though without formally ruling out any class of algorithms): the onset of \emph{shattering} (in the sense of above) roughly coincides with the point above which known efficient algorithms break down. The first rigorous link between the solution space geometry and formal algorithmic hardness is formed through the OGP framework introduced by Gamarnik and Sudan~\cite{gamarnik2014limits}. This framework leverages insights from statistical physics; at a high level it asserts (w.h.p.) the non-existence of a certain cluster of near-optimal solutions, which we refer to as a `forbidden structure'. For the $p$-spin model, the classical OGP~\cite{chen2019suboptimality,gamarnikjagannath2021overlap,gamarnik2020low} asserts that the region of overlaps  for any two near-optima $\bs,\bs'\in\Sigma_n$ is disconnected: there exists $0<\nu_1<\nu_2<1$ such that $n^{-1}\ip{\bs}{\bs'}\in[0,\nu_1]\cup [\nu_2,1]$. Since then the OGP framework was instrumental in establishing lower bounds for many other average-case models, including random CSPs~\cite{gamarnik2017performance,bresler2021algorithmic}, optimization over random graphs~\cite{gamarnik2017,gamarnik2020low,wein2020optimal}, number partitioning problem~\cite{gamarnik2021algorithmic}, symmetric Ising perceptron~\cite{gamarnik2022algorithms,gamarnik2023geometric}, and the $p$-spin model~\cite{gamarnikjagannath2021overlap,gamarnik2020low,gamarnik2021circuit,huang2021tight,huang2023algorithmic}. See also~\cite{gamarnik2021overlap} for a survey on OGP. Several of these subsequent works leveraged more involved forbidden structures, see below.
\paragraph{Multi OGP} The work~\cite{gamarnik2014limits} introducing the OGP framework considers the problem of finding a large independent set in sparse random graphs on $n$ vertices with average degree $d$. While the largest independent set in this model\footnote{In the double limit, $n\to\infty$ followed by $d\to\infty$.} is asymptotically of size $2\frac{\log d}{d}n$~\cite{frieze1992independence,bayati2010combinatorial}, the best known efficient algorithm finds an independent set of size only $\frac{\log d}{d}n$. This \texttt{SCG} was addressed in~\cite{gamarnik2014limits,gamarnik2017}; they showed that any pair of independent sets of size larger than $(1+1/\sqrt{2})\frac{\log d}{d}n$ exhibits the OGP, and subsequently, local algorithms fail to find a large independent set of size above $(1+1/\sqrt{2})\frac{\log d}{d}n$. The extra $1/\sqrt{2}$ factor was removed by Rahman and Vir{\'a}g~\cite{rahman2017local}, who showed the presence of the OGP for a forbidden structure involving many independent sets all the way down to the algorithmic $\frac{\log d}{d}n$ threshold. We refer this approach as the multi OGP ($m$-OGP). This approach is useful in obtaining (nearly) tight algorithmic lower bounds for various other average-case models. Of particular interest to us is the \emph{symmetric $m$-OGP} which asserts the non-existence of $m$-tuples of  near-optimal solutions whose pairwise overlaps are approximately the same. This version of OGP was introduced in~\cite{gamarnik2017performance} to obtain nearly tight lower bounds against the class of sequential local algorithms for the random Not-All-Equal $k$-SAT model. Subsequent works leveraged similar symmetric $m$-OGP to establish nearly tight lower bounds against stable algorithms for the symmetric Ising perceptron~\cite{gamarnik2022algorithms} and to establish (non-tight) lower bounds far above the existential value for the random number partitioning problem~\cite{gamarnik2021algorithmic}.

Recently, more sophisticated and asymmetric versions of $m$-OGP were proposed. These versions involve more intricate forbidden patterns (where the $i{\rm th}$ solution has `intermediate' overlap with the first $i-1$ solutions for $2\le i\le m$); they were crucial in establishing tight lower bounds against low-degree polynomials for random graphs~\cite{wein2020optimal} and the random $k$-SAT~\cite{huang2021tight}. For the $p$-spin model, classical OGP regarding pairs (described above) fails to establish tight hardness guarantees~\cite{gamarnikjagannath2021overlap,gamarnik2020low}. Huang and Sellke circumvented this issue by introducing a very clever version of the OGP consisting of an ultrametric tree of solutions~\cite{huang2021tight,huang2023algorithmic}. Dubbed as the \emph{branching OGP}, this framework yielded tight lower bounds against the class of Lipschitz algorithms, see~\cite{huang2021tight} for the description of this class. In this paper, we show that when $p$ is large, one can in fact establish the presence of the much simpler symmetric $m$-OGP.
\paragraph{Ensemble OGP} An idea emerged in~\cite{chen2019suboptimality} is to consider pairs that are near-optimal with respect to correlated instances. For the $p$-spin model, this corresponds to considering $\bs_1,\bs_2\in\Sigma_n$ such that $H_i(\bs_i) = n^{-\frac{p+1}{2}}\langle\boldsymbol{J}^{(i)},\bs_i^{\otimes p}\rangle\ge E$ for some energy proxy $E$, where $\boldsymbol{J}^{(1)},\boldsymbol{J}^{(2)}\in(\R^n)^{\otimes p}$ are correlated copies.  Dubbed as the \emph{ensemble OGP}, this property combined with $m$-OGP (hence ensemble $m$-OGP) is proven quite powerful in ruling out virtually any stable algorithm~\cite{gamarnik2020low,gamarnikjagannath2021overlap,wein2020optimal,gamarnik2021algorithmic,gamarnik2022algorithms,bresler2021algorithmic,huang2021tight,huang2023algorithmic}. Our focus in the present paper is also on the ensemble version of the symmetric $m$-OGP described above.
\paragraph{Paper Organization} The rest of the paper is organized as follows. We provide all of our main results in Section~\ref{sec:main-results}. In particular, see Section~\ref{sec:concentration-number-of-sol} for a certain concentration result regarding the number of solutions and a corollary regarding the ground-state value for large $p$;  Section~\ref{sec:clustering} for a clustering result in the landscape of the Hamiltonian; 
Section~\ref{sec:shattering} for our shattering result; and Section~\ref{sec:m-ogp} for our $m$-OGP result. 
We provide complete proofs of our results in Section~\ref{sec:pfs}.
\paragraph{Notation} We close this section with a brief list of notation. For any set $A$, denote its cardinality by $|A|$. Given any event $E$, denote its indicator by $\ind\{E\}$. For $n\in\N$, $\Sigma_n$ denotes the discrete cube $\{-1,1\}^n$. For any $\bs,\bs'\in\Sigma_n$, $d_H(\bs,\bs')$ denotes their Hamming distance: $d_H(\bs,\bs')=\sum_{1\le i\le n}\ind\{\bs(i)\ne \bs'(i)\}$. Given any $\mathcal{C},\mathcal{C}'\subset \Sigma_n$, ${\rm dist}(\mathcal{C},\mathcal{C}')$ denotes their distance, $\min_{\bs\in\mathcal{C},\bs'\in\mathcal{C}'}d_H(\bs,\bs')$. For any $v,v'\in \R$, $\ip{v}{v'}$ denotes their inner product, $\ip{v}{v'}=\sum_{1\le i\le n}v(i)v'(i)$. For any $r>0$, $\log_r(\cdot)$ and $\exp_r(\cdot)$ respectively denote the logarithm and exponential functions base $r$. When $r=e$, we denote the former by $\ln(\cdot)$ and the latter by $\exp(\cdot)$. For any $p\in[0,1]$, $h(p)=-p\log_2p-(1-p)\log_2(1-p)$ denotes the binary entropy function logarithm base 2. For any $n$, $I_n$ denotes the $n\times n$ identity matrix. Given any $\boldsymbol{\mu}\in\R^k$ and $\Sigma\in\R^{k\times k}$, $\mathcal{N}(\boldsymbol{\mu},\Sigma)$ denotes the multivariate normal distribution in $\R^k$ with mean $\boldsymbol{\mu}$ and covariance $\Sigma$. Given a matrix $\mathcal{M}$, $\|\mathcal{M}\|_F$, $\|\mathcal{M}\|_2$, and $|\mathcal{M}|$ denote, respectively, the
Frobenius norm, the spectral norm, and the determinant of $\mathcal{M}$. We employ standard Bachmann-Landau asymptotic notation throughout, e.g. $\Theta(\cdot),O(\cdot),o(\cdot),\Omega(\cdot)$ and $\omega(\cdot)$, where the underlying asymptotics is (often) with respect to $n\to\infty$.  Whenever a confusion is possible, we reflect the underlying asymptotics as a subscript. We omit all floor/ceiling operators for simplicity.
\section{Main Results}\label{sec:main-results}
In what follows, we denote by $H(\bs)$ the Hamiltonian for the pure $p$-spin model per~\eqref{eq:p-spin-Hamiltonian}, where the subscripts $n$ and $p$ are dropped for simplicity.  
Our main results are now in order. (We note that the values of constants such as $P(\epsilon),P^*,\hat{P}$ may change from line to line.)
\subsection{Concentration of Number of Solutions}\label{sec:concentration-number-of-sol}
Fix any $0<\epsilon<1$ and let 
\begin{align}\label{eq:s-eps} 
    \mathcal{S}(\epsilon)&\triangleq \Bigl\{\bs\in\Sigma_n: H(\bs)\ge (1-\epsilon)\sqrt{2\ln 2}\Bigr\}.
\end{align}
Our results are crucially based on the following tight concentration property for $|\mathcal{S}(\epsilon)|$, which we believe is of potential independent interest.
\begin{proposition}\label{prop:2nd-mom-est}
    For any $\epsilon>0$, there exists a $P(\epsilon)\in\mathbb{N}$ such that the following holds. Fix any $p\ge P(\epsilon)$. Then, as $n\to\infty$
    \[
    \mathbb{E}\left[|\mathcal{S}(\epsilon)|^2\right] = \mathbb{E}\bigl[|\mathcal{S}(\epsilon)|\bigr]^2\bigl(1+o_n(1)\bigr).
    \]
\end{proposition}
See Section~\ref{sec:pf-proposition} for the proof. Several remarks are in order. 

The proof of Proposition~\ref{prop:2nd-mom-est} involves a certain sum over all pairs $(\bs,\bs')\in\Sigma_n\times \Sigma_n$ with $H(\bs),H(\bs')\ge (1-\epsilon)\sqrt{2\ln 2}$. In order to study this sum, we employ various delicate estimates on binomial coefficients as well as a tail bound for bivariate normal random variables---note that $(H(\bs),H(\bs'))$ is a bivariate normal. Importantly, our argument shows that for $p$ large, the `dominant' contribution to the second moment comes from pairs $(\bs,\bs')$ that are nearly orthogonal, $n^{-1}\left|\ip{\bs}{\bs'}\right|\le n^{-O(1)}$.  We found it rather surprising that when $p$ is large, the second moment calculation is tight.

Using Paley-Zygmund inequality~\eqref{eq:paley-zygmund}, Proposition~\ref{prop:2nd-mom-est} immediately yields that for any fixed $\epsilon>0$ and large $p$, $|\mathcal{S}(\epsilon)|$ concentrates around its mean, $|\mathcal{S}(\epsilon)|\sim \mathbb{E}\bigl[|\mathcal{S}(\epsilon)|\bigr]$. This concentration property will be useful in our results to follow.

As an immediate corollary of this we obtain a simple proof that the ground-state energy for the Ising $p$-spin model is around $\sqrt{2\ln 2}$ for large $p$.
\begin{Corollary}\label{thm:ground-state-val}
    For any $\epsilon>0$, there is a $P(\epsilon)\in\mathbb{N}$ such that for $p\geq P(\epsilon)$,
    \[
    \mathbb{P}\left[(1-\epsilon)\sqrt{2\ln 2}\le \max_{\bs\in\Sigma_n}H(\bs)\le \sqrt{2\ln 2}\right]\ge 1-o_n(1).
    \]
\end{Corollary}
See Section~\ref{pf:ground-state-val} for the proof. The upper bound appearing in Corollary~\ref{thm:ground-state-val} follows from a simple application of the \emph{first moment method}. The lower bound, on the other hand, is also a simple implication of the \emph{second moment method} and Proposition~\ref{prop:2nd-mom-est}.

\begin{remark}
    Unlike Corollary~\ref{thm:ground-state-val}, it appears that Talagrand's results in~\cite{talagrand2000rigorous} yield an approximation for the ground-state energy only up to a factor of $1/2$. In particular, an argument based on the fact $\lim_{n\to\infty} \mathbb{E}[\ln Z_\beta]/n = \ln 2 + \beta^2/2$ for $\beta<\beta_p$ per~\eqref{eq:RSB} and the bound $\ln Z_\beta/n \le \ln 2 +\beta \OPT$, gives that for any $\epsilon>0$, there is a $P^*\in \N$ such that for $p\ge P^*$, $H^*\ge (1-\epsilon)\sqrt{2\ln 2}/2$, which is off by a factor of 1/2. This is due to the fact that the expression above for the (log) partition function is valid only up to $\beta_p$, which is asymptotically $\sqrt{2\ln 2}$, thus leading to a factor $1/2$ gap between the upper and lower bounds on the ground state values.
\end{remark}
\subsection{Clustering in the Landscape of Hamiltonian}\label{sec:clustering}
In this section, we establish a certain clustering property regarding the Hamiltonian~\eqref{eq:p-spin-Hamiltonian}. To that end, we rely on the following rather simple version of the Overlap Gap Property regarding pairs.
\begin{definition}\label{def:2-ogp}
Given $0<\nu_1<\nu_2<1$ and an $A\subset \Sigma_n$, $A$ is said to exhibit \emph{$(\nu_1,\nu_2)$-Overlap Gap Property},  $(\nu_1,\nu_2)$-OGP in short, if for any $\bs_1,\bs_2\in A$, either $d_H(\bs_1,\bs_2)\le \nu_1 n$ or $d_H(\bs_1,\bs_2)\ge \nu_2 n$.
\end{definition}
We employ Definition~\ref{def:2-ogp} when $A=\mathcal{S}(\epsilon)$ for some $\epsilon>0$. 
We now define the following notion of clustering.
\begin{definition}\label{def:CLUSTER}
    Let $A\subseteq\Sigma_n$. A collection of subsets of $A$, $(\mathcal{C}_\ell)_{\ell=1}^L$, is called a \emph{$(\nu_1,\nu_2)$-clustering} of $A$ if it is a partition of $A$ that satisfies:
    \begin{itemize}
\item[(a)] For any $1\le i \le L$ and any $\bs,\bs’\in\mathcal{C}_i$, $d_H(\bs,\bs’)\le \nu_1 n$.
\item[(b)] For any $1\le i < j \le L$ and any $\bs\in \mathcal{C}_i$, $\bs’\in\mathcal{C}_j$, $d_H(\bs,\bs’)\ge \nu_2 n$. 
\end{itemize}
\end{definition}
The next result shows that when a set $A$ exhibits the $(\nu_1,\nu_2)$-OGP in the sense of Definition~\ref{def:2-ogp} for some $0<\nu_1<\frac12\nu_2<\nu_2<1$, then it admits a unique $(\nu_1,\nu_2)$-clustering.
 \begin{proposition}\label{prop:ogp-implies-clustering}
Suppose that $A\subset \Sigma_n$ exhibits the $(\nu_1,\nu_2)$-OGP for some $0<\nu_1<\frac12\nu_2<\frac12$. Then, there exists a unique 
$(\nu_1,\nu_2)$-clustering of $A$,  $(\mathcal{C}_\ell)_{\ell=1}^L$.
\end{proposition}
Proposition~\ref{prop:ogp-implies-clustering} is originally due to~\cite{achlioptas2011solution}; its proof on Section~\ref{sec:pf-ogp-implies-clustering} is reproduced from~\cite{anschuetz2023combinatorial}. Note that the clusters $\mathcal{C}_i$ are well-separated: for any distinct $\mathcal{C}_i,\mathcal{C}_j$, ${\rm dist}(\mathcal{C}_i,\mathcal{C}_j)=\Omega(n)$. We highlight that under Proposition~\ref{prop:ogp-implies-clustering}, a single cluster (i.e.\,$L=1$) is still possible. 

Our next main result shows that for any $0<\epsilon<1-\frac{1}{\sqrt{2}}$ and any large $p$, $\mathcal{S}(\epsilon)$ appearing in~\eqref{eq:s-eps} breaks down into exponentially many well-separated clusters, i.e.\,$L=2^{\Omega(n)}$. 
\begin{theorem}\label{thm:clustered}
For any $0<\epsilon<1-\frac{1}{\sqrt{2}}$, there exists $\hat{P}\in\N$, $\nu_1,\nu_2$ with $0<\nu_1<\frac12\nu_2<\nu_2<1$ and $c_1>c_2>0$ such that the following holds. Fix any $p\ge \hat{P}$.
\begin{itemize}
    \item[(a)] The set $\mathcal{S}(\epsilon)$ exhibits $(\nu_1,\nu_2)$-OGP w.p.\,at least $1-e^{-\Theta(n)}$ as $n\to\infty$. 
    \item [(b)] Let $\mathcal{C}_1,\dots,\mathcal{C}_L$ be the $(\nu_1,\nu_2)$-clustering of $\mathcal{S}(\epsilon)$ per Proposition~\ref{prop:ogp-implies-clustering}. Then, with probability at least $1-e^{-\Theta(n)}$ as $n\to\infty$,  $|\mathcal{S}(\epsilon)|\ge 2^{c_1n}$ and 
    $\max_{1\le i\le L}|\mathcal{C}_i|\le 2^{c_2 n}$. In particular, $L=2^{\Omega(n)}$. 
\end{itemize}
Furthermore, the same holds also for the set $\mathcal{S}(\epsilon)\setminus\mathcal{S}(\epsilon')$, where $0<\epsilon'<\epsilon<1-\frac{1}{\sqrt{2}}$ are arbitrary. 
\end{theorem}
See below for the proof sketch and Section~\ref{sec:pf-clustered} for the complete proof. 
Later in Theorem~\ref{thm:shattering}, we leverage Theorem~\ref{thm:clustered} to show that the Ising pure $p$-spin model exhibits a strong notion of shattering mentioned earlier for certain temperatures in the replica symmetric region.
\paragraph{Proof Sketch for Theorem~\ref{thm:clustered}} We first show that for any $\epsilon$ in the given range, $\mathcal{S}(\epsilon)$ exhibits $(\nu_1,\nu_2)$-OGP for some $\nu_1<\frac12\nu_2$. This is based on a (rather simple) first moment argument. Using Proposition~\ref{prop:ogp-implies-clustering}, we take the $(\nu_1,\nu_2)$-clustering of $A$ and denote it by $\mathcal{C}_1,\dots,\mathcal{C}_L$, $L\ge 1$. It remains to show $L=2^{\Omega(n)}$. To that end, we use Proposition~\ref{prop:2nd-mom-est} and the second moment method to show that for some $c_1>0$, $|\mathcal{S}(\epsilon)|\ge 2^{c_1n}$ w.h.p. We then show, using the first moment method, that the number of pairs $(\bs_1,\bs_2)\in\mathcal{S}(\epsilon)\times \mathcal{S}(\epsilon)$ with $n^{-1}d_H(\bs_1,\bs_2)\le \nu_1$ is at most $2^{c'n}$ for some $c'>0$. 
Now observe that the number of all such pairs is $\sum_{i\le L}\binom{|\mathcal{C}_i|}{2}$, so
\[
\frac14 \max_{i\le L}|\mathcal{C}_i|^2 \le \sum_{i\le L}\binom{|\mathcal{C}_i|}{2} \le 2^{c'n},
\]
which immediately yields $\max_i |\mathcal{C}_i|\le 2^{c_2n+1}$ for $c_2=c'/2$. We then verify $c'<2c_1$, so that $c_1>c_2$. With this and the fact $|\mathcal{S}(\epsilon)| = \sum_{1\le i\le L}|\mathcal{C}_i|$, we immediately obtain $L=2^{\Omega(n)}$, as claimed.  
\subsection{Shattering}\label{sec:shattering}

In this section, we establish that for any $\sqrt{\ln 2}<\beta<\sqrt{2\ln 2}$ and any large enough $p$, the Ising pure $p$-spin model exhibits shattering.
We begin by first providing a notion of shattering.
 Let $d_H$ denote Hamming distance in $\Sigma_n$. For two sets $A,B$, denote their distance by ${\rm dist}(A,B)=\min_{x\in A,y\in B} d_H(x,y)$, and for a set $A$ denote its diameter by ${\rm diam}(A)=\max_{x,y\in A}d(x,y)$.
\begin{definition}\label{def:shattering}
    We say that the Gibbs measure is $(a,b)$-\emph{shattered} 
at inverse temperature $\beta>0$ if there exists constants $c,c'>0$ such that w.h.p.\,as $n\to\infty$ (w.r.t.\,$\boldsymbol{J}\in(\R^n)^{\otimes p}$), there exists non-empty subsets  
$(\mathcal{C}_\ell)_{\ell=1}^L\subset \Sigma_n$
called 
clusters such that 
the following holds:
\begin{itemize}
\item[(a)] There are exponentially many clusters:
\[
\frac{1}{n} \log_2 L \geq c.
\]
\item[(b)] The clusters are confined and well-separated, that is 
\[
\max_{1\leq i<j\leq L} {\rm diam}(\mathcal{C}_i) \leq a n \qquad \text{ and }\qquad 
\min_{1\le i<j\le L}{\rm dist}(\mathcal{C}_i,\mathcal{C}_j)=b n.
\]
\item[(c)] Each cluster is sub-dominant, namely has exponentially small Gibbs mass: \[
\max_{1\le i\le L}\mu_\beta(\mathcal{C}_i)\le \exp(-c'n).
\]
\item[(d)] The clusters collectively contain all but a vanishing fraction of total Gibbs mass:
\[
\mu_\beta\bigl(\cup_{1\le i\le L}\mathcal{C}_i\bigr)=1-o_n(1).
\]
\end{itemize}
\end{definition}
Our result is as follows. 
\begin{theorem}\label{thm:shattering}
    For any $\sqrt{\ln 2}<\beta<\sqrt{2\ln 2}$ and any small enough $\kappa>0$, there exists a $P^*\in\N$, $0<\nu_1<\frac12\nu_2<\nu_2<1$ and $c,c'>0$ such that the following holds. For any $p\ge P^*$, the Gibbs measure is $(2\nu_1,\nu_2)$-shattered at inverse temperature $\beta$. In particular,  the clusters are a $(\nu_1,\nu_2)$-clustering of $\{\bs\in\Sigma_n:\bigl|H(\bs)-\beta\bigr|\le \kappa\sqrt{2\ln 2}\}$ in the sense of Definition~\ref{def:CLUSTER}. 
\end{theorem}
See below for the proof sketch and Section~\ref{sec:shatttering-pf} for the complete proof. Namely, for any $\beta\in(\sqrt{\ln 2},\sqrt{2\ln 2})$ and any large enough $p$, the Ising pure $p$-spin model exhibits shattering. That is, there exists exponentially many well-separated clusters $\mathcal{C}_i$, $1\le i\le L$, such that with respect to the Gibbs distribution $\mu_\beta(\cdot)$ at inverse temperature $\beta$: (a) each cluster is \emph{sub-dominant}, i.e.\,contains an exponential small fraction of total Gibbs mass, and (b) the clusters collectively contain all but a vanishing fraction of Gibbs mass.
Our proof (modulo straightforward modifications) adapts also to the REM. That is, for any $\beta$ in the range above, the REM is also exhibits the shattering in the sense of above. The value $\sqrt{2\ln 2}$ corresponds to the critical temperature for the REM, see~\cite[Theorem~3.1]{bovier2009short}. 
\paragraph{Proof Sketch for Theorem~\ref{thm:shattering}} We first apply Theorem~\ref{thm:clustered} to establish that the set $\mathbb{D}(\beta,\kappa)=\{\bs\in\Sigma_n:H(\bs)\in[\beta-\kappa\sqrt{2\ln 2},\beta+\kappa\sqrt{2\ln 2}]\}$ partitions into exponentially many well-separated clusters. To verify  that each cluster sub-dominant, we rely on the estimates regarding the size of each cluster per Theorem~\ref{thm:clustered}. The most involved part of the proof is to show that the clusters collectively contain (w.h.p.) all but a vanishing fraction of Gibbs mass. We establish this by showing that for any $\beta<\sqrt{2\ln 2}$ and any fixed large enough $p$, the partition function $Z_\beta$ is dominated, w.h.p., by configurations $\bs\in\mathbb{D}(\beta,\kappa)$. Our result in fact proves a stronger conclusion that $\mu_\beta\bigl(\mathbb{D}(\beta,\kappa)^c\bigr)\le \exp(-\Theta(n))$. This is crucially based on the concentration property, Proposition~\ref{prop:2nd-mom-est} above.  For details, see Section~\ref{sec:shatttering-pf}. 

\subsection{Multi Overlap Gap Property ($m$-OGP)}\label{sec:m-ogp}
Equipped with the ground-state value per Corollary~\ref{thm:ground-state-val}, a natural algorithmic question is whether a near ground-state can be found efficiently. As we discussed in the introduction, the $p$-spin model exhibits the Overlap Gap Property (OGP), which is a barrier for large classes of algorithms. In particular, Huang and Sellke~\cite{huang2021tight,huang2023algorithmic} established that the $p$-spin model exhibits a rather sophisticated version of the OGP, dubbed as the branching OGP; they subsequently obtained tight lower bounds against the class of Lipschitz algorithms. 

The existing proofs for the OGP for spin glasses however are very technical; in particular they rely on the Parisi formula. In this section, we show that for large $p$, one can in fact consider a simpler multi OGP and establish its presence using rather elementary tools, namely the \emph{first moment method}. We begin by formalizing the set of $m$-tuples we investigate. 
\begin{definition}\label{def:admit}
    Let $m\in\N$, $0<\gamma<1$, $0<\eta<\xi<1$ and $\mathcal{I}\subset[0,\frac{\pi}{2}]$. Denote by $S(\gamma,m,\xi,\eta,\mathcal{I})$ the set of all $m$-tuples $\bs^{(t)}\in\Sigma_n,1\le t\le m$, that satisfy the following:
    \begin{itemize}
        \item {\bf $\gamma$-Optimality:} There exists $\tau_1,\dots,\tau_m\in\mathcal{I}$ such that 
        \[
        n^{-\frac{p+1}{2}}\sum_{1\le i_1,\dots,i_p\le n}\widehat{J}^{(t)}_{i_1,\dots,i_p}(\tau_t)\bs^{(t)}_{i_1}\cdots\bs^{(t)}_{i_p}\ge \gamma\sqrt{2\ln 2},\quad \forall 1\le t\le m,
        \]
        where for any $1\le t\le m$ and $\tau\in[0,\frac{\pi}{2}]$,
        \[
        \widehat{J}^{(t)}_{i_1,\dots,i_p}(\tau) = \cos(\tau)J^{(0)}_{i_1,\dots,i_p}+\sin(\tau)J^{(t)}_{i_1,\dots,i_p}
        \]
        for i.i.d.\,$J^{(t)}_{i_1,\dots,i_p}\sim \cN(0,1)$, $0\le t\le m$ and $1\le i_1,\dots,i_p\le n$.
        \item {\bf Overlap Constraint:} For any $1\le t<\ell\le m$, $n^{-1}\ip{\bs^{(t)}}{\bs^{(\ell)}}\in[\xi-\eta,\xi]$.
    \end{itemize}
\end{definition}
Definition~\ref{def:admit} regards $m$-tuples that are near-optimal with respect to correlated Hamiltonians. The term $\gamma$ quantifies the near-optimality, and the set $\mathcal{I}$ is used for defining correlated instances. It is necessary to consider correlated instances to obstruct stable algorithms, see below. The terms $\xi,\eta$ collectively define an overlap constraint, where one can think of $\xi\gg \eta$. Namely, the $m$-tuples $\bs_1,\dots,\bs_m\in\Sigma_n$ considered in Definition~\ref{def:admit} are nearly equidistant with pairwise Hamming distance about $n\frac{1-\xi}{2}$.

Our next main result establishes that for $p$ large, the Ising pure $p$-spin model exhibits the symmetric version of the ensemble $m$-OGP.
\begin{theorem}\label{thm:m-ogp}
For any $m\in\N$ and any $\gamma>1/\sqrt{m}$, there exists $0<\eta<\xi<1$, $c>0$, and $P^*\in\N$ such that the following holds. Fix any $p\ge P^*$ and any $\mathcal{I}\subset [0,\frac{\pi}{2}]$ with $|\mathcal{I}|\le 2^{cn}$. Then, 
\[
\mathbb{P}\bigl[S(\gamma,m,\xi,\eta,\mathcal{I})\ne\varnothing\bigr]\le e^{-\Theta(n)}
\]
as $n\to\infty$. 
\end{theorem}
Our proof is based on the \emph{first moment method}. More specifically, we let $M=\bigl|S(\gamma,m,\xi,\eta,\mathcal{I})\bigr|$ and show that $\mathbb{E}[M]= e^{-\Theta(n)}$ 
 %
  for suitable $\xi,\eta$ and large $p$. Our argument is based on a tail bound regarding multivariate normal random vectors (reproduced below as Theorem~\ref{thm:multiv-tail}), as well as Slepian's Gaussian comparison inequality (also reproduced as Lemma~\ref{lemma:Slepian})~\cite{slepian1962one} to address correlated instances. For large $p$, our argument shows that the exponent of a certain probability term regarding an $m$-dimensional multivariate normal random vector is close to that of $m$ i.i.d.\,standard normals. See Section~\ref{sec:pf-m-ogp} for the complete proof.

In conclusion, for any $m\in\N$ and any $\gamma>1/\sqrt{m}$, the pure $p$-spin model exhibits ensemble symmetric $m$-OGP above $\gamma\sqrt{2\ln 2}$ for all large enough $p$.
\paragraph{Symmetric $m$-OGP for large $p$ and Algorithmic Threshold in REM} We observe from Theorem~\ref{thm:m-ogp} that the onset of the symmetric $m$-OGP (for a suitable $m\in\N$) approaches to 0 as $p$ grows. Curiously, the value $0$ is the algorithmic threshold for the REM as we now elaborate. For this, we rely on a prior work by Addario-Berry and Maillard~\cite{addario2020algorithmic} which studies a continuous version of the REM called CREM. Using the notation of~\cite[Theorem~1.1]{addario2020algorithmic}, we observe that REM corresponds to CREM with $A(t) = 0\cdot \ind\{t\in[0,1)\} + \ind\{t=1\}$, i.e.\,the $a(t)$ appearing therein is a delta mass at 1. Applying now~\cite[Theorem~1.1]{addario2020algorithmic}, we find that the algorithmic threshold for the REM is at 0. (We thank Brice Huang for this argument.) Now, recall from the introduction that the $m$-OGP marks the threshold at which certain  classes of algorithms break down (also see below). In light of these facts, we arrive at an interesting conclusion: as $p\to\infty$, the algorithmic threshold for the Ising $p$-spin model, as prescribed by the symmetric $m$-OGP, approaches to that of REM, namely the value zero. 

\paragraph{Algorithmic Lower Bounds} We now return to the algorithmic problem of efficiently finding a near ground-state. In the context of $p$-spin models, we consider algorithms $\A:(\R^n)^{\otimes p}\to \Sigma_n$  accepting a $\boldsymbol{J}\in (\R^n)^{\otimes p}$ with i.i.d.\,standard normal entries and a $\gamma>0$ as their inputs and returning a $\A(\boldsymbol{J}) = \bs_{\rm ALG}\in\Sigma_n $ such that $H(\bs_{\rm ALG})\ge \gamma \max_{\bs \in \Sigma_n}H(\bs)$, ideally w.h.p.\,as $n\to\infty$. The ensemble $m$-OGP established in Theorem~\ref{thm:m-ogp} is a rigorous barrier for certain powerful classes of algorithms exhibiting input stability\footnote{Informally, an algorithm $\A$ is stable if for any two inputs $\boldsymbol{J}$ and $\boldsymbol{J}'$ with a small $\|\boldsymbol{J}-\boldsymbol{J}'\|$, the outputs $\A(\boldsymbol{J})$ and $\A(\boldsymbol{J}')$ are close in the Hamming distance. For a more formal definition, see~\cite{gamarnik2021algorithmic,gamarnik2022algorithms}.}. The classes of algorithms against which the OGP is a provable barrier include low-degree polynomials~\cite{gamarnik2020low,wein2020optimal}, AMP~\cite{gamarnikjagannath2021overlap}, low-depth Boolean circuits~\cite{gamarnik2021circuit} and overlap concentrated algorithms~\cite{huang2021tight,huang2023algorithmic}. The latter class includes, in particular, $O(1)$ iterations of AMP and Langevin dynamics run for $O(1)$ time, see~\cite{huang2021tight}. It is worth noting that the best known polynomial-time algorithm for the $p$-spin model can in fact be implemented as an AMP algorithm run for $O(1)$ iterations. Using Theorem~\ref{thm:m-ogp}, one can establish that for any $\gamma>0$, there is a $P_\gamma\in\N$ such that for any $p\ge P_\gamma$, there do not exist a sufficiently stable algorithm that finds (w.h.p.) a $\bs_{\rm ALG}$ with $H(\bs_{\rm ALG})\ge \gamma\sqrt{2\ln 2}$. This can be done by directly adapting the techniques of, e.g.~\cite{gamarnik2021algorithmic,gamarnik2022algorithms}. We refer the reader to these citations
for details.

\section{Proofs}\label{sec:pfs}
In this section, we provide complete proofs of all of our main results.
\subsection{Auxiliary Results}
We collect several useful auxiliary results below.
\paragraph{Probabilistic Estimates} The first result is the well-known Gaussian tail bound: for $x>0$,
\begin{equation}\label{eq:gaussian-tail}
    \frac{\exp(-x^2/2)}{\sqrt{2\pi}}\left(\frac1x-\frac{1}{x^3}\right) \le \mathbb{P}[\cN(0,1)\ge x]\le  \frac{\exp(-x^2/2)}{x\sqrt{2\pi}}.
\end{equation}
In particular when $x=\omega_n(1)$,~\eqref{eq:gaussian-tail} yields
\[
\mathbb{P}\bigl[\cN(0,1)\ge x\bigr]  = \frac{\exp(-x^2/2)}{x\sqrt{2\pi}}(1+o_n(1)).
\]
We next record the following bivariate normal tail bound.  
\begin{lemma}\label{lemma:biv-tail}
Let $(Z,Z_\rho)$ be a bivariate normal random vector with $Z,Z_\rho\sim \cN(0,1)$ and $\mathbb{E}[Z\cdot Z_\rho]=\rho\in(-1,1)$. Then for any $t>0$, 
\[
\mathbb{P}\bigl[Z>t,Z_\rho>t\bigr]\le \frac{(1+\rho)^2}{2\pi t^2\sqrt{1-\rho^2}}\exp\left(-\frac{t^2}{1+\rho}\right).
\]
\end{lemma}

Our $m$-OGP result, Theorem~\ref{thm:m-ogp}, is crucially based on the following tail bound regarding multivariate normal random vectors. It is originally due to Savage~\cite{savage1962mills}; the version we cite below is from~\cite{hashorva2003multivariate,hashorva2005asymptotics}. 
\begin{theorem}\label{thm:multiv-tail}
Let $\boldsymbol{X}\in\R^d$ be a centered multivariate normal random vector with non-singular covariance matrix $\Sigma\in\R^{d\times d}$ and $\boldsymbol{t}\in\R^d$ be a fixed threshold. Suppose that $\Sigma^{-1}\boldsymbol{t}>\boldsymbol{0}$ entrywise. Then,
\[
1-\ip{1/(\Sigma^{-1}\boldsymbol{t})}{\Sigma^{-1}(1/(\Sigma^{-1}\boldsymbol{t})} \le \frac{\mathbb{P}[\boldsymbol{X}\ge \boldsymbol{t}]}{\varphi_{\boldsymbol{X}}(\boldsymbol{t})\prod_{i\le d}\ip{e_i}{\Sigma^{-1}\boldsymbol{t}}}\le 1,
\]
where $e_i\in\R^d$ is the $i{\rm th}$ unit vector and $\varphi_{\boldsymbol{X}}(\boldsymbol{t})$ is the multivariate normal density evaluated at $\boldsymbol{t}$:
\[
\varphi_{\boldsymbol{X}}(\boldsymbol{t}) = (2\pi)^{-d/2}|\Sigma|^{-1/2}\exp\left(-\frac{\boldsymbol{t}^T\Sigma^{-1}\boldsymbol{t}}{2}\right)\in\R^+.
\]
\end{theorem}
We employ Theorem~\ref{thm:multiv-tail} for the case where $\Sigma$ is `close to identity'.

Several of our main results are based on the \emph{second moment method}. To that end, we recall Paley-Zygmund inequality. Let $Z$ be a random variable that is almost surely non-negative and ${\rm Var}(Z)<\infty$. Then, for any $\theta\in[0,1]$,
\begin{equation}\label{eq:paley-zygmund}
    \mathbb{P}\bigl[Z>\theta \mathbb{E}[Z]\bigr]\ge (1-\theta)^2 \frac{\mathbb{E}[Z]^2}{\mathbb{E}[Z^2]}.
\end{equation}
\paragraph{Auxiliary Results from Linear Algebra} Our $m$-OGP result, Theorem~\ref{thm:m-ogp}, requires several linear-algebraic arguments. We begin by reminding the reader the Sherman-Morrison matrix inversion formula~\cite{sherman1950adjustment}:
\begin{theorem}\label{thm:sm}
Let $A\in\R^{n\times n}$ be an invertible matrix and $u,v\in\R^n$ be column vectors. Then, $(A+uv^T)^{-1}$ exists iff $1+v^T A^{-1}u\ne 0$ and the inverse is given by the formula
\[
\bigl(A+uv^T\bigr)^{-1} = A^{-1} - \frac{A^{-1}uv^T A^{-1}}{1+v^T A^{-1}u}.
\]
\end{theorem}
Finally, we record Wielandt-Hoffman inequality~\cite{hoffman1953variation}, see also~\cite[Corollary 6.3.8]{horn2012matrix}.
\begin{theorem}\label{thm:hw}
Let $A,A+E\in\R^{n\times n}$  be two symmetric matrices with respective eigenvalues
\[
\lambda_1(A)\ge \lambda_2(A)\ge\cdots\ge\lambda_n(A)\quad\text{and}\quad \lambda_1(A+E)\ge \lambda_2(A+E)\ge\cdots\ge \lambda_n(A+E).
\]
Then
\[
\sum_{1\le i\le n}\left(\lambda_i(A+E)-\lambda_i(A)\right)^2\le \|E\|_F^2.
\]
\end{theorem}
\subsection{Proof of Proposition~\ref{prop:2nd-mom-est}}\label{sec:pf-proposition}
\begin{proof}[Proof of Proposition~\ref{prop:2nd-mom-est}]
Fix any $\epsilon\in(0,1)$ and recall $\mathcal{S}(\epsilon)$ from~\eqref{eq:s-eps}.
\paragraph{First moment estimate.} 
Fix any $\bs\in\Sigma_n$ and recall $H(\bs)\sim \cN(0,1/n)$. 
Using~\eqref{eq:gaussian-tail},  we get
\[
\mathbb{P}\bigl[H(\bs)\ge (1-\epsilon)\sqrt{2\ln 2}\bigr] = \mathbb{P}\bigl[\cN(0,1)\ge (1-\epsilon)\sqrt{2n \ln 2}\bigr]=\frac{2^{-n(1-\epsilon)^2}}{(1-\epsilon)\sqrt{4\pi n\ln 2}}\bigl(1+o_n(1)\bigr). 
\]
Hence, 
\begin{equation}\label{eq:1st-mom-est}
   \mathbb{E}[|\mathcal{S}(\epsilon)|] = \frac{2^{n-n(1-\epsilon)^2}}{(1-\epsilon)\sqrt{4\pi n\ln 2}}(1+o_n(1)).
\end{equation}

\paragraph{Second moment estimate.} Our next focus is on the second moment. Observe that
\begin{align}
    &\bigl\{\alpha : \alpha = n^{-1}\ip{\bs_1}{\bs_2}\text{ for some }\bs_1,\bs_2 \in \Sigma_n\bigr\} \nonumber
    \\
    &= \{\alpha:\alpha\in[-1,1],n-\alpha n\equiv 0\pmod{2}\}\triangleq O\label{eq:admissible-overlap}.
\end{align}
Moreover, for any $\alpha\in O$, the number of pairs $(\bs_1,\bs_2)\in\Sigma_n\times \Sigma_n$ with $n^{-1}\ip{\bs_1}{\bs_2} = \alpha$ is $2^n\cdot \binom{n}{\frac{1-\alpha}{2}n}$. Next, fix any $\bs_1,\bs_2$ such that $n^{-1}\ip{\bs_1}{\bs_2} = \alpha$. Then,
\begin{align}
    \mathbb{P}\bigl[\bs_1,\bs_2\in \mathcal{S}(\epsilon)\bigr] &= \mathbb{P}\bigl[H(\bs_1)\ge (1-\epsilon)\sqrt{2\ln 2},H(\bs_2)\ge (1-\epsilon)\sqrt{2\ln 2}\bigr] \nonumber\\
    &=\mathbb{P}\bigl[Z\ge (1-\epsilon)\sqrt{2n\ln 2},Z_\alpha\ge (1-\epsilon) \sqrt{2n\ln 2} \bigr]\triangleq p(\alpha)\label{eq:p-alpha}
\end{align}
where $Z,Z_\alpha \sim \cN(0,1)$ with $\mathbb{E}[Z\cdot Z_\alpha] =\alpha^p$. In particular, using Lemma~\ref{lemma:biv-tail}, we bound $p(\alpha)$ as
\begin{equation}\label{eq:p-alpha-up-bd}
    p(\alpha)\le \frac{1}{C_1^2 n}\cdot \frac{\left(1+\alpha^p\right)^2}{\sqrt{1-\alpha^{2p}}}\exp_2\left(-\frac{2n(1-\epsilon)^2}{1+\alpha^p}\right),
\end{equation}
where 
\begin{equation}\label{eq:C_1}
    C_1 = (1-\epsilon)\sqrt{4\pi\ln 2}.
\end{equation}
Using~\eqref{eq:p-alpha}, we obtain
\begin{equation}\label{eq:2nd-mom-1}
    \mathbb{E}[|\mathcal{S}(\epsilon)|^2] =2^n\sum_{\alpha \in O}\binom{n}{\frac{1-\alpha}{2}n}p(\alpha).
\end{equation}
We investigate~\eqref{eq:2nd-mom-1} based on size of $\alpha$.
\paragraph{Large $|\alpha|$.} Fix any arbitrary $\alpha^* = \alpha^*(\epsilon)\in(0,1)$ such that 
\begin{equation}\label{eq:alpha-star}
-1 + h\left(\frac{1-\alpha^*}{2}\right) + (1-\epsilon)^2 <0,
\end{equation}
where $h(p)=-p\log_2 p -(1-p)\log_2(1-p)$ is the binary entropy function. An $\alpha^*$ satisfying~\eqref{eq:alpha-star} indeed exists, as $\lim_{\alpha\to 1}-1 + h\left(\frac{1-\alpha}{2}\right) + (1-\epsilon)^2  = -1+(1-\epsilon)^2<0$ and $\alpha\mapsto -1 + h\left(\frac{1-\alpha}{2}\right) + (1-\epsilon)^2$ is continuous. 
Note that for any $\alpha\in O$,~\eqref{eq:gaussian-tail} yields
\begin{equation}\label{eq:trivial}
    p(\alpha)\le \mathbb{P}\bigl[\cN(0,1)\ge (1-\epsilon)\sqrt{2n\ln 2}\bigr]\le \frac{1}{C_1\sqrt{n}}2^{-n(1-\epsilon)^2}.
\end{equation}
Hence, 
\begin{align}
    &\frac{1}{\mathbb{E}[|\mathcal{S}(\epsilon)|]^2} 2^n \sum_{\alpha \in O:|\alpha|\ge \alpha^*} \binom{n}{\frac{1-\alpha}{2}n} p(\alpha) \nonumber\\
    &\le \frac{C_1^2 n}{2^{n-2n(1-\epsilon)^2}} \sum_{\alpha \in O:|\alpha|\ge \alpha^*} \binom{n}{\frac{1-\alpha}{2}n} p(\alpha)\label{eq:use-1st-mom}\\
    &\le \frac{C_1 n^{O(1)}}{2^{n-n(1-\epsilon)^2}} \binom{n}{\frac{1-\alpha^*}{2}n}\label{eq:use-trivial} \\
&=\exp_2\left(n\left(-1+h\left(\frac{1-\alpha^*}{2}\right) +(1-\epsilon)^2\right)+O(\log_2 n)\right) \label{eq:use-stirling}\\
&=e^{-\Theta(n)}\label{eq:use-a-star},
\end{align}
where~\eqref{eq:use-1st-mom} uses the first moment estimate~\eqref{eq:1st-mom-est},~\eqref{eq:use-trivial} uses~\eqref{eq:trivial},~\eqref{eq:use-stirling} follows from Stirling's approximation, $\binom{n}{n\rho} = \exp_2\bigl(nh(\rho) + O(\log_2 n)\bigr)$ valid for all $\rho\in(0,1)$, and lastly,~\eqref{eq:use-a-star} uses~\eqref{eq:alpha-star}. So,
\begin{equation}\label{eq:large-alpha-est}
    \frac{1}{\mathbb{E}[|\mathcal{S}(\epsilon)|]^2} \cdot 2^n\sum_{\alpha \in O:|\alpha|\ge \alpha^*} \binom{n}{\frac{1-\alpha}{2}n}p(\alpha) = \exp\bigl(-\Theta(n)\bigr).
\end{equation}
\paragraph{Small $|\alpha|$.}  
We now fix a $\iota\in(0,\frac12)$ and focus on $\alpha$ such that $n^{-\iota}\le |\alpha|\le \alpha^*$. We collect several auxiliary results.
\begin{lemma}\label{lem:bin-coeff-tight}
For any $1\le k\le n-1$,
\[
\binom{n}{k}\le \sqrt{\frac{n}{2\pi k(n-k)}}\exp_2\left(nh\left(\frac{k}{n}\right)\right).
\]
In particular, for $k=n\frac{1-\alpha}{2}$ with $\alpha\ne \pm 1$, we obtain
\[
\binom{n}{k}\le \sqrt{\frac{2}{\pi n(1-\alpha^2)}}\exp_2\left(nh\left(\frac{1-\alpha}{2}\right)\right).
\]
\end{lemma}
\begin{proof}[Proof of Lemma~\ref{lem:bin-coeff-tight}]
    The argument below is due to~\cite[Exercise~5.8] {gallager1968information}. Using Stirling's formula,
    \[
    \sqrt{2\pi n}\left(\frac{n}{e}\right)^n \exp\left(\frac{1}{12n+1}\right)<n!< \sqrt{2\pi n}\left(\frac{n}{e}\right)^n \exp\left(\frac{1}{12n}\right)
    \]
    for every $n\ge 1$. In particular,
    \[
    n! = \sqrt{2\pi n}\left(\frac{n}{e}\right)^n \exp\bigl(z_n\bigr),
    \]
    where $(z_n)_{n\ge 1}$ is a decreasing sequence with $0<z_n<\frac{1}{12n}$. Plugging this, we obtain
    \begin{align*}
            \binom{n}{k} &= \sqrt{\frac{n}{2\pi k(n-k)}}\left(\frac{n}{k}\right)^k \left(\frac{n}{n-k}\right)^{n-k}\exp\left(z_n - z_k-z_{n-k}\right) \\
            &\le \sqrt{\frac{n}{2\pi k(n-k)}}\exp_2\left(nh\left(\frac{k}{n}\right)\right),
        \end{align*}
        using the fact $z_n - z_k<0$ and $z_{n-k}>0$. 
\end{proof}
 We next recall the Taylor series for the binary entropy function in a neighborhood of $1/2$ (see e.g.~\cite[Equation~12]{ordentlich2015minimum}):
\begin{equation}\label{eq:bin-ent-taylor}
    h\left(\frac{1-\alpha}{2}\right)= 1-\frac{1}{2\ln 2}\sum_{n\ge 1}\frac{\alpha^{2n}}{n(2n-1)}.
\end{equation}
Using~\eqref{eq:bin-ent-taylor}, we obtain
\begin{equation}\label{eq:bin-ent-taylor2}
    h\left(\frac{1-\alpha}{2}\right)\le 1 -\frac{\alpha^2}{2\ln 2} - \frac{\alpha^4}{12\ln 2}.
\end{equation}
Combining Lemma~\ref{lem:bin-coeff-tight} with~\eqref{eq:bin-ent-taylor2} and recalling $\alpha^2\le (\alpha^*)^2$, we immediately obtain
\begin{equation}\label{eq:bin-coeff-upbd}
    \binom{n}{n\frac{1-\alpha}{2}}\le \sqrt{\frac{2}{\pi n(1-(\alpha^*)^2)}} \exp_2\left(n - \frac{n\alpha^2}{2\ln 2}-\frac{n\alpha^4}{12\ln 2}\right).
\end{equation}
Next, using~\eqref{eq:p-alpha-up-bd}, we obtain that for any $\alpha$ with $n^{-\iota}\le |\alpha|\le \alpha^*<1$, 
\begin{equation}\label{eq:p-alpha-up-bd3}
    p(\alpha)\le \frac{1}{C_1^2 n}\frac{4}{\sqrt{1-(\alpha^*)^{2p}}}\exp_2\left(-\frac{2n(1-\epsilon)^2}{1+\alpha^p}\right),
\end{equation}
where $C_1 = (1-\epsilon)\sqrt{4\pi \ln 2}$ per~\eqref{eq:C_1}. We now combine~\eqref{eq:1st-mom-est},~\eqref{eq:bin-coeff-upbd}, and~\eqref{eq:p-alpha-up-bd3} to conclude
\begin{align}\label{eq:auxil-bd3}
    \frac{1}{\mathbb{E}[|\mathcal{S}(\epsilon)|]^2} \cdot 2^n\binom{n}{n\frac{1-\alpha}{2}}p(\alpha)\le C(\alpha^*,p)\frac{1+o_n(1)}{\sqrt{n}}\exp_2\left(-\frac{n\alpha^2}{2\ln 2}-\frac{n\alpha^4}{12\ln 2} + \frac{2n(1-\epsilon)^2 \alpha^p}{1+\alpha^p}\right),
\end{align}
where 
\begin{equation}\label{eq:C-zeta-p}
    C(\alpha^*,p) = \frac{4}{\sqrt{1-(\alpha^*)^{2p}}}\cdot \sqrt{\frac{2}{\pi (1-(\alpha^*)^2)}}.
\end{equation}
We next establish the following lemma.
\begin{lemma}\label{lemma:neg}
Fix $\alpha^*>0$ and take $p$ such that
\begin{equation}\label{eq:p-sat-this}
    p\ge \ln\left(\frac{24\ln 2}{(\alpha^*)^4}\right)\cdot \ln\left(\frac{1}{\alpha^*}\right)^{-1}.
\end{equation}
Then for every $-\alpha^*\le \alpha\le \alpha^*$ and every $0<\epsilon\leq 1$,
    \[
    -\frac{\alpha^4}{12\ln 2} + \frac{2(1-\epsilon)^2\alpha^p}{1+\alpha^p}\le 0.
    \]
\end{lemma}
\begin{proof}
    Note that the claim is immediate if $\alpha^p\le 0$ as $1+\alpha^p>0$, so assume $\alpha^p = |\alpha|^p>0$. Moreover, \eqref{eq:p-sat-this} yields $(\alpha^*)^{-p+4}\ge 24\ln 2$. Next, as
 $0\le |\alpha|\le (\alpha^*)$, we obtain
    \[
    \left(\frac{1}{|\alpha|}\right)^{p-4} \ge (\alpha^*)^{-p+4}\ge 24\ln 2 \implies \alpha^4 \ge 24|\alpha|^p\ln 2.
    \]
    From here, we establish Lemma~\ref{lemma:neg} as
    \[
    -\frac{\alpha^4}{12\ln 2}+\frac{2(1-\epsilon)^2\alpha^p}{1+\alpha^p}\le |\alpha|^p\left(-2+\frac{2(1-\epsilon)^2}{1+|\alpha|^p}\right)\le 0.
    \]
\end{proof}
Combining Lemma~\ref{lemma:neg} with~\eqref{eq:auxil-bd3}, we obtain that provided $p$ satisfies~\eqref{eq:p-sat-this}, we have
\begin{equation}\label{eq:auxil-bd4}
    \frac{1}{\mathbb{E}[|\mathcal{S}(\epsilon)|]^2} \cdot 2^n\binom{n}{n\frac{1-\alpha}{2}}p(\alpha)\le C(\alpha^*,p) \frac{1+o_n(1)}{\sqrt{n}}\exp\left(-\frac{n\alpha^2}{2}\right).
\end{equation}
Now, suppose $n^{-\iota}\le |\alpha|\le \alpha^*$, so that $n^{-2\iota}\le \alpha^2 \le (\alpha^*)^2$. Then, 
\begin{equation}\label{eq:auxil-bd5}
    \exp\left(-\frac{n\alpha^2}{2}\right)\le \exp\left(-\frac12 n^{1-2\iota}\right).
\end{equation}
Lastly, we recall $|O| = n^{O(1)}$ for the set $O$ defined in~\eqref{eq:admissible-overlap}. Equipped with this, we combine~\eqref{eq:auxil-bd4} and~\eqref{eq:auxil-bd5} to obtain
\begin{align}
    \frac{1}{\mathbb{E}[|\mathcal{S}(\epsilon)|]^2}\cdot 2^n\sum_{\substack{\alpha \in O\\ n^{-\iota}\le |\alpha|\le \alpha^*}}\binom{n}{n\frac{1-\alpha}{2}}p(\alpha)&\le 2C(\alpha^*,p)\cdot n^{O(1)} \exp\left(-\Theta(n^{1-2\iota})\right) 
    \nonumber  \\
    &\le \exp\left(-\Theta\left(n^{1-2\iota}\right)\right)\label{eq:small-alpha-est}. 
\end{align}

\paragraph{Vanishing $|\alpha|$.} Our last focus is on $\alpha$ with $-n^{-\iota}\le \alpha\le n^{-\iota}$. We will establish that the `dominant' contribution to~\eqref{eq:2nd-mom-1} comes from such $\alpha$. We collect several estimates. Using~\eqref{eq:p-alpha-up-bd}, we get
\begin{equation}\label{eq:p-alpha-vanishing-O}
  p(\alpha)\le \frac{1}{C_1^2 n} \cdot \frac{(1+n^{-\iota p})^2}{\sqrt{1-n^{-2\iota p}}}\exp_2\left(-\frac{2n(1-\epsilon)^2}{1+\alpha^p}\right).
\end{equation}
Next, we show
\begin{equation}\label{eq:auxil-bd2}
    \sup_{\alpha \in[-n^{-\iota},n^{-\iota}]} \frac{\alpha^p}{1+\alpha^p} \le n^{-\iota p}.
\end{equation}
If $\alpha^p<0$ the using $1+\alpha^p\ge 0$, we have $\frac{\alpha^p}{1+\alpha^p}\le 0$. On the other hand, if $\alpha^p\ge 0$ then $\alpha^p = |\alpha|^p\le n^{-\iota p}$, which together with the fact $t\mapsto \frac{t}{t+1}$ is increasing on $[0,\infty)$ yields $\frac{\alpha^p}{1+\alpha^p}\le \frac{n^{-\iota p}}{1+n^{-\iota p}}\le n^{-\iota p}$. These facts collectively establish~\eqref{eq:auxil-bd2}.

Next, for $C_1=(1-\epsilon)\sqrt{4\pi \ln 2}$, we have
\begin{align}
    &\frac{1}{\mathbb{E}[|\mathcal{S}(\epsilon)|]^2} \cdot 2^n\sum_{\alpha \in O\cap [-n^{-\iota},n^{-\iota}]}\binom{n}{n\frac{1-\alpha}{2}} p(\alpha) \nonumber\\
    &\le\frac{C_1^2n(1+o_n(1))}{2^{2n - 2n(1-\epsilon)^2}}2^n \frac{1}{C_1^2 n} \cdot \frac{(1+n^{-\iota p})^2}{\sqrt{1-n^{-2\iota p}}}\sum_{\alpha \in O\cap [-n^{-\iota},n^{-\iota}]}\binom{n}{n\frac{1-\alpha}{2}}\exp_2\left(-\frac{2n(1-\epsilon)^2}{1+\alpha^p}\right) \label{eq:auxil-111}\\
    &\le (1+o_n(1))\frac{(1+n^{-\iota p})^2}{\sqrt{1-n^{-2\iota p}}} \sum_{\alpha \in O\cap [-n^{-\iota},n^{-\iota}]}\binom{n}{n\frac{1-\alpha}{2}}\exp_2\left(-n+\frac{2n\alpha^p(1-\epsilon)^2}{1+\alpha^p}\right)\label{eq:auxil-222}\\
    &\le (1+o_n(1))\frac{(1+n^{-\iota p})^2}{\sqrt{1-n^{-2\iota p}}}\exp_2\left(2(1-\epsilon)^2 n^{1-\iota p}\right) \sum_{\alpha \in O\cap [-n^{-\iota},n^{-\iota}]}2^{-n}\binom{n}{n\frac{1-\alpha}{2}} \label{eq:auxil-3} \\
    &\le (1+o_n(1))\frac{(1+n^{-\iota p})^2}{\sqrt{1-n^{-2\iota p}}}\exp_2\left(2(1-\epsilon)^2 n^{1-\iota p}\right).\label{eq:auxil-333}
\end{align}
Here,~\eqref{eq:auxil-111} is obtained by combining $\mathbb{E}[|\mathcal{S}(\epsilon)|]$ per~\eqref{eq:1st-mom-est} and the upper bound~\eqref{eq:p-alpha-vanishing-O} on $p(\alpha)$;~\eqref{eq:auxil-222} follows from simple algebra; ~\eqref{eq:auxil-3}  follows from~\eqref{eq:auxil-bd2}; and~\eqref{eq:auxil-333} follows from the fact
\[
\sum_{\alpha \in O\cap [-n^{-\iota},n^{-\iota}]}2^{-n}\binom{n}{n\frac{1-\alpha}{2}}\le 1.
\]
Now, provided $p>\frac{1}{\iota}$, we have $n^{1-\iota p} = n^{-\Theta(1)}=o_n(1)$. This, together with~\eqref{eq:auxil-333} yields 
\begin{equation}\label{eq:vanishing-alpha-est}
    \frac{1}{\mathbb{E}[|\mathcal{S}(\epsilon)|]^2} \cdot 2^n\sum_{\alpha \in O\cap [-n^{-\iota},n^{-\iota}]} \binom{n}{n\frac{1-\alpha}{2}}p(\alpha) \le 1+o_n(1),
\end{equation}
for $p>1/\iota$.
\paragraph{Combining everything.} Note that $\mathbb{E}[|\mathcal{S}(\epsilon)|^2] \ge \mathbb{E}[|\mathcal{S}(\epsilon)|]^2$ by Jensen's inequality. Moreover, provided 
\[
p\ge \max\left\{\ln\left(\frac{24\ln 2}{(\alpha^*)^4}\right)\cdot \ln\left(\frac{1}{\alpha^*}\right)^{-1},\frac1\iota\right\},
\]
we obtain by combining~\eqref{eq:large-alpha-est},
~\eqref{eq:small-alpha-est} and~\eqref{eq:vanishing-alpha-est} that
\begin{align*}
    \mathbb{E}[|\mathcal{S}(\epsilon)|^2]&\le \mathbb{E}[|\mathcal{S}(\epsilon)|]^2\left(e^{-\Theta(n)} + e^{-\Theta(n^{1-2\iota})} + 1+o_n(1)\right)\\
    &=\mathbb{E}[|\mathcal{S}(\epsilon)|]^2(1+o_n(1)).
\end{align*}
This completes the proof of Proposition~\ref{prop:2nd-mom-est}.

\end{proof}
\subsection{Proof of Corollary~\ref{thm:ground-state-val}}\label{pf:ground-state-val}
Fix any $\epsilon>0$. Using Proposition~\ref{prop:2nd-mom-est}, we obtain that there exists a $P$ such that the following holds. Fix any $p\ge P$. Then,
\[
\mathbb{E}\bigl[|\mathcal{S}(\epsilon)|^2\bigr] =(1+o_n(1)) \mathbb{E}\bigl[|\mathcal{S}(\epsilon)|\bigr]^2.
\]
Using Paley-Zygmund inequality~\eqref{eq:paley-zygmund}, we get
\[
\mathbb{P}\bigl[\mathcal{S}(\epsilon)\ne \varnothing\bigr] = \mathbb{P}\bigl[|\mathcal{S}(\epsilon)|>0\bigr]\ge \frac{\mathbb{E}\bigl[|\mathcal{S}(\epsilon)|\bigr]^2}{\mathbb{E}\bigl[|\mathcal{S}(\epsilon)|^2\bigr]} = 1-o_n(1).
\]
Hence,
\begin{equation}\label{eq:gs-low-bd}
\mathbb{P}\left[\max_{\bs\in\Sigma_n}H(\bs)\ge (1-\epsilon)\sqrt{2\ln 2}\right] \ge \mathbb{P}\bigl[\mathcal{S}(\epsilon)\ne\varnothing\bigr] = 1-o_n(1).
\end{equation}
As for the upper bound, we have
\begin{align}
    \mathbb{E}\left[\left|\left\{\bs \in\Sigma_n:H(\bs)\ge \sqrt{2\ln 2}\right\}\right|\right]&=2^n \mathbb{P}\bigl[\cN(0,1)\ge \sqrt{2n\ln 2}\bigr] \label{eq:hamilton-is-gauss}\\
    &\le 2^n \frac{1}{\Theta(\sqrt{n})}\exp\left(-n\ln 2\right) \label{eq:gs-gs-tail}\\
    &=\Theta\left(\frac{1}{\sqrt{n}}\right)\label{eq:gs-up-bd}
\end{align}
where~\eqref{eq:hamilton-is-gauss} uses the linearity of expectation together with the fact $H(\bs)\sim \cN(0,n^{-1})$ for any $\bs\in\Sigma_n$,~\eqref{eq:gs-gs-tail} uses the Gaussian tail bound~\eqref{eq:gaussian-tail}. 

Combining~\eqref{eq:gs-low-bd} and~\eqref{eq:gs-up-bd} via a union bound, we establish Corollary~\ref{thm:ground-state-val}.
\subsection{Proof of Proposition~\ref{prop:ogp-implies-clustering}}\label{sec:pf-ogp-implies-clustering}
Define a relation $\sim$ on $A$ such that $\bs\sim \bs’$ iff $n^{-1}d_H(\bs,\bs’)\le \nu_1$. We verify that $\sim$ is an equivalence relation, which partitions $A$ into disjoint equivalence classes. Denoting them by $\mathcal{C}_i$, we then establish Proposition~\ref{prop:ogp-implies-clustering}. note that $\sim$ is clearly symmetric and reflexive, hence it suffices to establish its transitivity. To that end, take any $\bs_1,\bs_2,\bs_3\in A$ with $\bs_1\sim \bs_2$ and $\bs_2\sim \bs_3$. We then have $n^{-1}d_H(\bs_1,\bs_2)\le \nu_1$ and $n^{-1}d_H(\bs_2,\bs_3)\le \nu_1$. So, $n^{-1}d_H(\bs_1,\bs_3)\le 2\nu_1<\nu_2$ by the triangle inequality. Since the set $A$ exhibits $(\nu_1,\nu_2)$-OGP, it follows that $n^{-1}d_H(\bs_1,\bs_3)\le \nu_1$, so $\bs_1\sim \bs_3$, as claimed.
\subsection{Proof of Theorem~\ref{thm:clustered}}\label{sec:pf-clustered}
Towards Theorem~\ref{thm:clustered}, we establish several auxiliary results. For any fixed $0<\nu_1<\nu_2<\frac12$, denote by $S(\nu_1,\nu_2,\epsilon)$ the set of all pairs $(\bs_1,\bs_2)\in \mathcal{S}(\epsilon)\times \mathcal{S}(\epsilon)$ such that $n^{-1}d_H(\bs_1,\bs_2)\in[\nu_1,\nu_2]$.

Our first auxiliary result shows that $\mathcal{S}(\epsilon)$ exhibits $(\nu_1,\nu_2)$-OGP for suitable $\nu_1,\nu_2$.
\begin{proposition}\label{prop:2-ogp}
    For any $\epsilon\in(0,1-1/\sqrt{2})$, there exists $P_{\rm OGP}\in\mathbb{N}$, $0<\nu_1<\frac{\nu_2}{2}<\frac14$ such that for any $p\ge P_{\rm OGP}$, 
    \[
\mathbb{P}\bigl[S(\nu_1,\nu_2,\epsilon)=\varnothing\bigr] = 1-\exp\bigl(-\Theta(n)\bigr).
    \]
That is, for any $p\ge P_{\rm OGP}$, $\mathcal{S}(\epsilon)$ exhibits $(\nu_1,\nu_2)$-OGP w.p.\,at least $1-\exp(-\Theta(n))$.
\end{proposition}
\begin{proof}[Proof of Proposition~\ref{prop:2-ogp}]
    Fix  $0<\epsilon<1-\frac{1}{\sqrt{2}}$. For $0<\nu_1<\nu_2<\frac12$ to be tuned,  we control $\mathbb{E}[|S(\nu_1,\nu_2,\epsilon)|]$.
    \paragraph{Counting term.} Note that,
    \begin{align}
        \Bigl|\bigl\{(\bs_1,\bs_2)\in\Sigma_n^2 : n^{-1}d_H(\bs_1,\bs_2)\in[\nu_1,\nu_2]\bigr\}\Bigr| &\le 2^n\sum_{k\in[n\nu_1,n\nu_2]\cap\mathbb{Z}}\binom{n}{k}\nonumber \\
        &\le \exp_2\Bigl(n+nh(\nu_2)+O(\log_2 n)\Bigr) \label{eq:count-2-ogp}.
    \end{align}
    \paragraph{Probability estimate.} Fix any $\bs_1,\bs_2\in\Sigma_n$ with $n^{-1}d_H(\bs_1,\bs_2)\in[\nu_1,\nu_2]$. We now bound the probability of the event $H(\bs_1),H(\bs_2)\ge (1-\epsilon)\sqrt{2\ln 2}$. Let $Z_i = \sqrt{n}H(\bs_i)$. Clearly $Z_i \sim \cN(0,1)$. Moreover, $\mathbb{E}[Z_1Z_2] = \mathcal{R}(\bs_1,\bs_2)^p$, where $\mathcal{R}(\bs_1,\bs_2) = n^{-1}\ip{\bs_1}{\bs_2}\in[1-2\nu_2,1-2\nu_1]$. Applying Lemma~\ref{lemma:biv-tail} with $t=(1-\epsilon)\sqrt{2n \ln 2}$, 
    \begin{align}
        \mathbb{P}\bigl[\min\{Z_1,Z_2\}\ge (1-\epsilon)\sqrt{2n\ln 2}\bigr]&\le \frac{(1+\mathcal{R}(\bs_1,\bs_2)^p)^2}{4\pi n(1-\epsilon)^2\ln 2}\exp_2\left(-\frac{2n(1-\epsilon)^2}{1+\mathcal{R}(\bs_1,\bs_2)^p}\right)\nonumber\\
        &\le \exp_2\left(-\frac{2n(1-\epsilon)^2}{1+(1-2\nu_1)^p}+O(\log_2 n)\right)\nonumber \\
        &\le \exp_2\left(-2n(1-\epsilon)^2\left(1-(1-2\nu_1)^p\right)+O(\log_2 n)\right)\nonumber
    \end{align}
    using the fact $\mathcal{R}(\bs_1,\bs_2)\in[1-2\nu_2,1-2\nu_1]$ with $\nu_2<\frac12$ and the trivial bound $1-x\le \frac{1}{1+x}$ valid for $x\ge -1$. Since $\bs_1,\bs_2$ are arbitrary, we obtain
    \begin{equation}\label{eq:prob-2ogp}
        \sup_{\substack{\bs_1,\bs_2\in \Sigma_n \\ n^{-1}d_H(\bs_1,\bs_2)\in[\nu_1,\nu_2]}}\mathbb{P}\bigl[\bs_1,\bs_2\in\mathcal{S}(\epsilon)\bigr]\le \exp_2\Bigl(-2n(1-\epsilon)^2\left(1-(1-2\nu_1)^p\right)+O(\log_2 n)\Bigr).
    \end{equation}
    \paragraph{Combining everything.} We now set $\nu_1 = \frac{\nu_2}{3}$, which automatically satisfies $2\nu_1<\nu_2$. With this, 
    \begin{align}\label{eq:first-mom-2ogp}
        \mathbb{E}[|S(\nu_1,\nu_2,\epsilon)|] \le \exp_2\Bigl(n\bigl(1+h(\nu_2)-2(1-\epsilon)^2 + (1-2\nu_2/3)^p\bigr)+O(\log_2 n)\Bigr).
    \end{align}
    As $0<\epsilon<1-\frac{1}{\sqrt{2}}$, there exists a $\nu_2\in(0,\frac12)$ and a $P_{\rm OGP}$ such that for all $p\ge P_{\rm OGP}$, 
    \[
    1+h(\nu_2)-2(1-\epsilon)^2 + (1-2\nu_2/3)^p<0.
    \]
    Consequently, $\mathbb{E}[|S(\nu_1,\nu_2,\epsilon)|]\le \exp_2\bigl(-\Theta(n)\bigr)$, so we conclude by Markov's inequality. 
\end{proof}
We next upper bound the expected number of solutions at distance at most $\nu_1$.
\begin{proposition}\label{prop:cluster-size}
    Fix any $\nu_1\in(0,\frac12)$. Then, as $n\to\infty$,
    \[
    \mathbb{E}\bigl[|S(0,\nu_1,\epsilon)|\bigr]\le \exp_2\bigl(cn + O(\log_2 n)\bigr),
    \]
    where
    \[
    c = \max\left\{1+h(\delta) -(1-\epsilon)^2,1+h(\nu_1)-2(1-\epsilon)^2+(1-2\delta)^p\right\}
    \]
    for any $0<\delta<\nu_1$.
\end{proposition}
\begin{proof}[Proof of Proposition~\ref{prop:cluster-size}]
    Fix any $\delta\in(0,\nu_1)$ and observe that
    \begin{equation}\label{eq:up-bd-auxil}
        |S(0,\nu_1,\epsilon)| \le |S(0,\delta,\epsilon)|+ |S(\delta,\nu_1,\epsilon)|.
        \end{equation}
    Note that the argument leading to~\eqref{eq:first-mom-2ogp} remains the same, hence
\begin{equation}\label{eq:2nd-part}
    \mathbb{E}\bigl[|S(\delta,\nu_1,\epsilon)|\bigr]\le \exp_2\Bigl(n\bigl(1+h(\nu_1)-2(1-\epsilon)^2 +(1-2\delta)^p\bigr)+O(\log_2 n)\Bigr).
\end{equation}
We next estimate $\mathbb{E}[|S(0,\delta,\epsilon)|]$. Note that the number of pairs $(\bs_1,\bs_2)$ with $n^{-1}d_H(\bs_1,\bs_2)\in[0,\delta]$ is at most $\exp_2\bigl(n+nh(\delta)+O(\log_2 n)\bigr)$. Now, for any fixed $\bs_1,\bs_2$, observe that
\begin{align*}
    \mathbb{P}\bigl[\bs_1,\bs_2\in\mathcal{S}(\epsilon)\bigr]&\le \mathbb{P}\bigl[\sqrt{n}H(\bs_1)\ge (1-\epsilon)\sqrt{2n\ln 2}\bigr] \\
    &\le \exp_2\bigl(-n(1-\epsilon)^2 +O(\log_2 n)\bigr),
\end{align*}
using the fact $\sqrt{n}H(\bs_1)\sim \cN(0,1)$ and the tail bound~\eqref{eq:gaussian-tail}. Combining these facts, we obtain
\begin{equation}\label{eq:first-pt}
    \mathbb{E}\bigl[|S(0,\delta,\epsilon)|\bigr]\le \exp_2\Bigl(n\bigl(1+h(\delta)-(1-\epsilon)^2\bigr)+O(\log_2 n)\Bigr).
\end{equation}
With this, we establish Proposition~\ref{prop:cluster-size} by combining~\eqref{eq:up-bd-auxil},~\eqref{eq:2nd-part}, and~\eqref{eq:first-pt}.
\end{proof}
Equipped with Propositions~\ref{prop:2-ogp} and~\ref{prop:cluster-size}, we now prove Theorem~\ref{thm:clustered}.
\begin{proof}[Proof of Theorem~\ref{thm:clustered}]
Fix $0<\epsilon<1-\frac{1}{\sqrt{2}}$.  Note that for any $\gamma>0$, Paley-Zygmund inequality~\eqref{eq:paley-zygmund} yields together with Proposition~\ref{prop:2nd-mom-est} and~\eqref{eq:1st-mom-est} that if $p\ge \max\{P^*,P_{\rm OGP}\}$ is fixed, then 
\begin{equation}\label{eq:S-eps-low-bd}
    \mathbb{P}\Bigl[|\mathcal{S}(\epsilon)| \ge \exp_2\bigl(n(\epsilon(2-\epsilon)-\gamma)\bigr)\Bigr] = 1-o_n(1).
\end{equation}
We now define
\begin{equation}\label{eq:Xi-eps}
    \Xi(\epsilon) \triangleq \frac{\min\left\{\epsilon^{10},(1-\epsilon)^{10}\right\}}{100}.
\end{equation}
and set
\begin{equation}\label{eq:gamma}
    \gamma = \min\left\{\frac{1-h(\nu_1)-2\Xi(\epsilon)}{5},\frac{1-(1-\epsilon)^2-2\Xi(\epsilon)}{5}\right\}.
\end{equation}
Note that from our choice of $\nu_1,\nu_2$ per Proposition~\ref{prop:2-ogp}, we have $v_2<\frac12$ and $v_1=\frac{v_2}{3}<\frac16$, so 
\[
1-h(\nu_1)-2\Xi(\epsilon)\ge 1-h(\frac16)-\frac{1}{100\cdot 2^{10}}>0, 
\]
as $\min\{\epsilon,1-\epsilon\}\le 1/2$ for $\epsilon\in[0,1]$. This ensures that $\gamma$ defined in~\eqref{eq:gamma} is not vacuous. Next, $1-(1-\epsilon)^2-2\Xi(\epsilon)-4\gamma>0$.  With this, we choose any $\delta\in(0,1/2)$ with the property that
\begin{equation}\label{eq:delta}
 h(\delta)<1-(1-\epsilon)^2 - 2\Xi(\epsilon)-4\gamma.
\end{equation}
With these, we let
\begin{equation}\label{eq:c-1}
    c_1 = \epsilon(2-\epsilon)-\gamma.
\end{equation}
Next, Proposition~\ref{prop:cluster-size} together with Markov's inequality yield
\[
\mathbb{P}\bigl[|S(0,\nu_1,\epsilon)|\ge \exp_2((c+\gamma)n)\bigr]\le 2^{-(c+\gamma)n}\mathbb{E}\bigl[|S(0,\nu_1,\epsilon)|\bigr]\le \exp_2\bigl(-\Theta(n)\bigr),
\]
where $c$ is the exponent appearing in Proposition~\ref{prop:cluster-size}. We set
$\widehat{c}_2 = c+\gamma$, apply Proposition~\ref{prop:ogp-implies-clustering}, and denote by $\mathcal{C}_\ell$, $1\le \ell \le L$, the $(\nu_1,\nu_2)$-clustering of $\mathcal{S}(\epsilon)$:
 \[
 \mathcal{S}(\epsilon) = \bigcup_{1\le \ell \le L}\mathcal{C}_\ell.
 \]
 Note that for any pair $(\sigma^1,\sigma^2)\in\mathcal{S}(0,\nu_1,\epsilon)$ we must have that both $\sigma^1$ and $\sigma^2$ are elements of the same $\mathcal{C}_i$ for some $i\leq L$. As such,
\[
\frac14 \max_{1\le \ell \le L}|\mathcal{C}_\ell|^2 \le \sum_{1\le \ell\le L}\binom{|\mathcal{C}_\ell|}{2} = |S(0,\nu_1,\epsilon)| \le 2^{\widehat{c}_2 n},
\]
so that
\[
\max_{1\le \ell\le L}|C_\ell|\le \exp_2\left(\frac12 \widehat{c}_2 n + 1\right)\le \exp_2\bigl(c_2 n\bigr),
\]
where 
\begin{equation}\label{eq:c-2}
    c_2 = \frac{c}{2} + \gamma.
\end{equation}
We also have w.h.p.\,
\[
2^{c_1 n}\le |\mathcal{S}(\epsilon)| = \left|\bigcup_{1\le \ell \le L}\mathcal{C}_\ell\right| \le \sum_{1\le \ell \le L}|\mathcal{C}_\ell|\le L\cdot 2^{c_2n} \implies L\ge 2^{(c_1-c_2)n}.
\]
We now verify $c_1>c_2$ for all $p$ large enough. In fact, we establish a stronger conclusion that
\begin{equation}\label{eq:stronger-concl}
    c_1-c_2 > \Xi(\epsilon),
\end{equation}
where $\Xi(\epsilon)$ is defined in~\eqref{eq:Xi-eps}. For this, it suffices to verify
\begin{align}
    1-(1-\epsilon)^2 -\gamma &> \frac12\left(1+h(\delta)-(1-\epsilon)^2\right) + \gamma +\Xi(\epsilon)\label{verify-1}\\
     1-(1-\epsilon)^2 -\gamma &> \frac12\left(1+h(\nu_1)-2(1-\epsilon)^2+(1-2\delta)^p\right) + \gamma +\Xi(\epsilon).\label{verify-2} 
\end{align}
for all $p$ large. Note that~\eqref{verify-1} is equivalent to $1-(1-\epsilon)^2 - 2\Xi(\epsilon)-4\gamma>h(\delta)$ which holds per~\eqref{eq:delta}. On the other hand~\eqref{verify-2} is equivalent to
\[
1-h(\nu_1) > 4\gamma+(1-2\delta)^p+2\Xi(\epsilon).
\]
Per~\eqref{eq:gamma}, 
\[
1-h(\nu_1)-4\gamma-2\Xi(\epsilon)>\frac{1-h(\nu_1)}{5}>0.
\]
As $1-2\delta\in(0,1)$, there exists a $P_2\in\mathbb{N}$ such that
\[
\frac{1-h(\nu_1)}{5}>(1-2\delta)^p,\quad\text{for all}\quad p\ge P_2.
\]
Taking $\widehat{P} = \max\{P^*,P_{\rm OGP},P_2\}$, $c_1$ as in~\eqref{eq:c-1} and $c_2$ as in~\eqref{eq:c-2}, we establish Theorem~\ref{thm:clustered}.
\paragraph{Extension to $\mathcal{S}(\epsilon)\setminus \mathcal{S}(\epsilon')$} We now outline the extension to $\mathcal{S}(\epsilon)\setminus \mathcal{S}(\epsilon')$, where $0<\epsilon'<\epsilon<1-\frac{1}{\sqrt{2}}$. As $\mathcal{S}(\epsilon)\setminus \mathcal{S}(\epsilon')\subseteq \mathcal{S}(\epsilon)$, both Proposition~\ref{prop:2-ogp} and Proposition~\ref{prop:cluster-size} remain valid. Furthermore, using Lemma~\ref{lem:F-conc} we also have
\[
|\mathcal{S}(\epsilon)| = |\mathcal{S}(\epsilon)\setminus \mathcal{S}(\epsilon')|\left(1+e^{-\Theta(n)}\right).
\]
With these, rest of the argument remains the same.
\end{proof}
\subsection{Proof of Theorem~\ref{thm:shattering}}\label{sec:shatttering-pf}
Fix any $\sqrt{\ln 2}<\beta<\sqrt{2\ln 2}$, and choose $\kappa>0$ sufficiently small, so that for
\[
\epsilon = 1 -\left(\frac{\beta}{\sqrt{2\ln 2}}-\kappa\right),
\]
$\kappa<\frac14 \Xi(\epsilon)$ for $\Xi(\epsilon)$ defined in~\eqref{eq:Xi-eps}.
Define
\begin{equation}\label{eq:D-of-eps}
        \mathbb{D}(\beta,\kappa) = \bigl\{\bs\in\Sigma_n: H(\bs)\in[\beta-\kappa\sqrt{2\ln 2},\beta+\kappa\sqrt{2\ln 2}]\bigr\}.
\end{equation}
We first establish the following concentration result.
\begin{lemma}\label{lem:F-conc}
Fix $1>\kappa_2>\kappa_1>0$ and let
\begin{equation}\label{eq:F-kappa}
    \mathcal{F}(\kappa_1,\kappa_2)\triangleq \left\{\bs\in\Sigma_n:\frac{H(\bs)}{\sqrt{2\ln 2}}\in[\kappa_1,\kappa_2]\right\}.  
\end{equation}
There exists a $P\in\mathbb{N}$ such that the following holds. Fix any $p\ge P$. Then, 
\[
\bigl|\mathcal{F}(\kappa_1,\kappa_2)\bigr| = \exp_2\bigl(n(1-\kappa_1^2) + O(\log_2 n)\bigr)
\]
w.h.p.\,as $n\to\infty$.
\end{lemma}
   \begin{proof}[Proof of Lemma~\ref{lem:F-conc}]
Observe that for $\mathcal{S}(\epsilon)$ defined in~\eqref{eq:s-eps}, 
\[
\mathcal{F}(\kappa_1,\kappa_2) = \mathcal{S}(1-\kappa_1)\setminus \mathcal{S}(1-\kappa_2).
\]
Using Proposition~\ref{prop:2nd-mom-est}, we obtain that there exists a $P$ such that for all $p\ge P$,
\[
\mathbb{E}\bigl[\bigl|\mathcal{S}(1-\kappa_i)\bigr|^2\bigr] = (1+o_n(1))\mathbb{E}\bigl[\bigl|\mathcal{S}(1-\kappa_i)\bigr|\bigr]^2,\quad \forall i\in\{1,2\}.
\]
Now fix arbitrary $C>c>0$. 
Using Paley-Zygmund inequality~\eqref{eq:paley-zygmund}, we have
\begin{equation}\label{eq:F-low-bd}
    \mathbb{P}\Bigl[\bigl|\mathcal{S}(1-\kappa_i)\bigr|>n^{-c}\mathbb{E}\bigl[\bigl|\mathcal{S}(1-\kappa_i)\bigr|\bigr]\Bigr]\ge (1-n^{-c})^2(1+o_n(1)) = 1-o_n(1).
\end{equation}
Furthermore, by Markov's inequality,
\begin{equation}\label{eq:F-up-bd}
    \mathbb{P}\Bigl[\bigl|\mathcal{S}(1-\kappa_i)\bigr|\ge n^C\mathbb{E}\bigl[\bigl|\mathcal{S}(1-\kappa_i)\bigr|\bigr]\Bigr]\le n^{-C}.
\end{equation}
Combining~\eqref{eq:F-low-bd} and~\eqref{eq:F-up-bd} and recalling
\[
\mathbb{E}[|\mathcal{S}(1-\kappa_i)|] = \exp_2\bigl(n-n\kappa_i^2 +O(\log_2 n)\bigr)
\]
per~\eqref{eq:1st-mom-est}, we obtain
\begin{equation}\label{eq:F-mid-bd}
    |\mathcal{S}(1-\kappa_i)| = \exp_2\bigl(n-n\kappa_i^2 + O(\log_2 n)\bigr),
\end{equation}
which immediately yields the claim.
\end{proof}
Next, we apply Theorem~\ref{thm:clustered} to $\mathbb{D}(\beta,\kappa)$ and obtain that there exists $0<\nu_1<\nu_2/2<\nu_2<1$ and a $P_2^*\in\N$ such that for any $p\ge P_2^*$, there exists clusters described in Theorem~\ref{thm:clustered} w.h.p.\,as $n\to\infty$. Assume $p\ge P^*=\max\{P_1^*,P_2^*\}$ and the clusters are $\mathcal{C}_1,\dots,\mathcal{C}_L$, so that
\begin{equation}\label{eq:energy-band}
    \bigcup_{1\le i\le L}\mathcal{C}_i=\mathbb{D}(\beta,\kappa)=\bigl\{\bs\in\Sigma_n:H(\bs)\in[\beta-\kappa\sqrt{2\ln 2},\beta+\kappa\sqrt{2\ln 2}]\bigr\}.
\end{equation}
We now verify the requirements of Theorem~\ref{thm:shattering}.

\paragraph{Verifying Exponentially Many Well-Separated Clusters} Let $c_1>c_2>0$ be as in Theorem~\ref{thm:clustered}, so that 
\[
\bigl|\cup_{1\le i\le L}\mathcal{C}_i\bigr|\ge 2^{c_1n}\quad\text{and}\quad \max_{1\le i\le L}|\mathcal{C}_i|\le 2^{c_2n}.
\]
Clearly $L\ge 2^{(c_1-c_2)n}=2^{\Omega(n)}$. This, together with~\eqref{eq:energy-band} verifies ${\rm (a)}$. As for the ${\rm (b)}$, we have from the construction of $\mathcal{C}_1,\dots,\mathcal{C}_L$ (see Section~\ref{sec:clustering}) that (a) for any $1\le i\le L$ and $\bs,\bs\in\Sigma_n$, $n^{-1}d_H(\bs,\bs')\le \nu_1$ and (b) for any $1\le i<j\le L$, $\bs\in\mathcal{C}_i$ and $\bs'\in\mathcal{C}_j$, $n^{-1}d_H(\bs,\bs')\ge \nu_2$.  
\paragraph{Verifying Sub-Dominance} We now verify ${\rm (c)}$: each cluster is \emph{sub-dominant}, that is $\max_i \mu_\beta(\mathcal{C}_i)\le e^{-\Theta(n)}$. 
Fix $1\le i\le L$ and note, using~\eqref{eq:energy-band}, that
\begin{equation}\label{eq:sub1}
    \mu_\beta(\mathcal{C}_i) \le \frac{2^{c_2n}\cdot e^{\beta(\beta+\kappa\sqrt{2\ln 2})n}}{Z_\beta} = \frac{1}{Z_\beta}\exp\bigl(c_2n\ln 2+ \beta(\beta+\kappa\sqrt{2\ln 2})n\bigr).
\end{equation}
Furthermore, 
\begin{equation}\label{eq:sub2}
    \mu_\beta\bigl(\cup_{1\le i\le L}\mathcal{C}_i\bigr)\ge \frac{2^{c_1n} e^{\beta(\beta-\kappa\sqrt{2\ln 2}) n}}{Z_\beta}=  \frac{1}{Z_\beta}\exp\bigl(c_1n\ln 2 + \beta(\beta-\kappa\sqrt{2\ln 2}) n\bigr).
\end{equation}
Combining~\eqref{eq:sub1} and~\eqref{eq:sub2}, we arrive at
\[
\frac{\mu_\beta(\mathcal{C}_i)}{\mu_\beta\bigl(\cup_{1\le i\le L}\mathcal{C}_i\bigr)}\le \exp\bigl(-n\bigl((c_1-c_2)\ln 2 - \beta \kappa\bigr)\bigr).
\]
Thus, it suffices to verify 
\[
(c_1-c_2)\ln 2>2\beta\kappa\sqrt{2\ln 2},
\]
which is straightforward by using the facts $\beta<\sqrt{2\ln 2}$, $\kappa<\frac14\Xi(\epsilon)$ and $c_1-c_2>\Xi(\epsilon)$ as we verified in~\eqref{eq:stronger-concl}. These establish sub-dominance. 

\subsection*{Verifying the Condition on the Gibbs Mass of $\cup_{1\le i\le L}\mathcal{C}_i$}
The last part is to verify ${\rm (d)}$, that is to establish
\[
\mu_\beta\bigl(\cup_{1\le i\le L}\mathcal{C}_i\bigr) = 1-o_n(1)
\]
w.h.p.\,as $n\to\infty$. Recall that $\mathbb{D}(\beta,\kappa)=\cup_{1\le i\le L}\mathcal{C}_i$. We now establish a stronger conclusion that
\begin{equation}\label{eq:auxil11111}
    \mu_\beta\bigl(\cup_{1\le i\le L}\mathcal{C}_i\bigr)  = \mu_\beta(\mathbb{D}(\beta,\kappa))\ge 1-e^{-\Theta(n)}.
\end{equation}
\begin{proposition}\label{prop:pf-dominated}
    For any $\beta<\sqrt{2\ln 2}$ and any $\kappa<1/\sqrt{2}$, there exists a $P^*\in\N$ such that the following holds. Fix any $p\ge P^*$. Then, 
    \[
    Z_\beta = Z_\beta\left[\frac{\beta}{\sqrt{2\ln 2}}-\kappa,\frac{\beta}{\sqrt{2\ln 2}}+\kappa\right](1+e^{-\Theta(n)}),
    \]
    where 
    \begin{equation}\label{eq:partition-fnc}
        Z_\beta[\kappa_1,\kappa_2] = \sum_{\bs : \kappa_1\sqrt{2\ln 2}\le H(\bs)\le \kappa_2\sqrt{2\ln 2}} e^{\beta n H(\bs)}.
    \end{equation}
\end{proposition}
Note that Proposition~\ref{prop:pf-dominated} immediately yields~\eqref{eq:auxil11111} since 
\[
\mu_\beta(\mathbb{D}(\beta,\kappa)) = \frac{Z_\beta\left[\frac{\beta}{\sqrt{2\ln 2}}-\kappa,\frac{\beta}{\sqrt{2\ln 2}}+\kappa\right]}{Z_\beta}.
\]
\begin{proof}[Proof of Proposition~\ref{prop:pf-dominated}]
 Using Lemma~\ref{lem:F-conc}, we immediately obtain that for any $0<\kappa_1<\kappa_2\le 1$, there exists a $P(\kappa_1,\kappa_2)\in\mathbb{N}$ such that for all fixed $p\ge P(\kappa_1,\kappa_2)$  and any $\beta>0$,
\begin{equation}\label{eq:up-bd-Z-rest}
    Z_\beta[\kappa_1,\kappa_2]\le \bigl|\mathcal{F}(\kappa_1,\kappa_2)\bigr|e^{\beta n \kappa_2\sqrt{2\ln 2}}=\exp\Bigl(n\Bigl((1-\kappa_1^2)\ln 2 + \beta \kappa_2\sqrt{2\ln 2}\Bigr) + O(\ln n)\Bigr)
\end{equation}
w.h.p.\,as $n\to\infty$. Additionally, for any $\kappa>0$ we have (deterministically) that  
\begin{equation*}
    Z(0,\kappa)\le 2^n e^{\beta n\kappa \sqrt{2\ln 2}}=\exp\left(n\ln 2+\beta n\kappa\sqrt{2\ln 2}\right).
\end{equation*}
In particular,~\eqref{eq:up-bd-Z-rest} holds for any $0\le \kappa_1<\kappa_2\le 1$. 
\paragraph{Discretization.}
We now fix a $\kappa$ to be tuned and discretize $[0,1]$ in the following way. Define the quadratic function \begin{equation}\label{eq:quadratic}
    \varphi(\gamma) = (1-\gamma^2)\ln 2+ \beta \gamma \sqrt{2\ln 2},
\end{equation}
and let
\begin{equation}\label{eq:gamma-star-I-star}
    \gamma^* = \frac{\beta}{\sqrt{2\ln 2}} = \argmax_{\gamma}\varphi(\gamma)\quad\text{and}\quad I^* = [\gamma^*-\kappa,\gamma^*+\kappa].
\end{equation}
For discretization, let $0=\tau_0<\tau_1<\cdots<\tau_Q = \gamma^*-\kappa$ with $\tau_i-\tau_{i-1}=\kappa^4$ for $1\le i\le Q-1$ and $\tau_Q-\tau_{Q-1}\le \kappa^4$. Set $I_i = [\tau_i,\tau_{i+1}]$ for $0\le i\le Q-1$. Furthermore, let $\gamma^*+\kappa=\eta_0<\eta_1<\cdots<\eta_L=1$ with $\eta_i-\eta_{i-1}=\kappa^4$ for $1\le i\le L-1$ and $\eta_L-\eta_{L-1}\le \kappa^4$. Then set $J_i = [\eta_i,\eta_{i+1}]$ for $0\le i\le Q-1$. These yield a discretization of $[0,1]\setminus I^*$ into $Q+L$ intervals of length at most $\kappa^4$.
\paragraph{Restricted Partition Functions.} Set 
\begin{align}
    Z_i &\triangleq Z_\beta[\tau_i,\tau_{i+1}],\quad 0\le i\le Q-1 \\
    \hat{Z}_i&\triangleq Z_\beta[\eta_i,\eta_{i+1}],\quad 0\le i\le L-1 \\
    Z_n&\triangleq Z_\beta[-\infty,0].
\end{align}
Additionally, set 
\begin{align*}
    Z^* &= Z_\beta[\gamma^*-\kappa,\gamma^*+\kappa].
\end{align*}
We trivially have $Z_n\le \exp(n\ln 2)=\exp(n\varphi(0))$. Next, using~\eqref{eq:up-bd-Z-rest}, we obtain that there exists $P_0,\dots,P_{Q-1}$ such that for any $0\le i\le Q-1$ and any $p\ge P_i$, 
\begin{align}
       Z_i &\le\exp\Bigl(n(1-\tau_i^2)\ln 2+n\beta \tau_{i+1}\sqrt{2\ln 2}+O(\ln n)\Bigr)\label{eq:use-Z-up-bd}\\
       &\le \exp\Bigl(n(1-\tau_i^2)\ln 2+n\beta \tau_{i}\sqrt{2\ln 2}+n\beta\kappa^4 \sqrt{2\ln 2}+O(\ln n)\Bigr)\label{eq:tau-i-bd}\\
       &=\exp\Bigl(n\varphi(\tau_i)+n\beta\kappa^4 \sqrt{2\ln 2}+O(\ln n)\Bigr)\label{eq:useful},
\end{align}
 w.h.p.; where~\eqref{eq:use-Z-up-bd} uses~\eqref{eq:up-bd-Z-rest},~\eqref{eq:tau-i-bd} uses the fact $\tau_{i+1}-\tau_i\le \kappa^4$ and the last equation recalls the definition of $\varphi$ per~\eqref{eq:quadratic}. Similarly, there exists $\hat{P}_0,\dots,\hat{P}_{L-1}$ such that for any $0\le i\le L-1$ and any $p\ge \hat{P}_i$, such that w.h.p.
 \begin{equation}\label{eq:useful2}
     \hat{Z}_i \le \exp\Bigl(n\varphi(\eta_i)+n\beta\kappa^4 \sqrt{2\ln 2}+O(\ln n)\Bigr).
 \end{equation}
 We now control $Z^*$. Using Lemma~\ref{lem:F-conc}, we obtain that there exists a $P^*\in\N$ such that for $p\ge P^*$,
 \begin{equation}\label{eq:Z-star-low}
      Z^* \ge Z_\beta[\gamma^*,\gamma^*+\kappa]\ge \exp\Bigl(n(1-(\gamma^*)^2)\ln 2 + n\beta\gamma^*\sqrt{2\ln 2}-O(\ln n)\Bigr) = \exp\Bigl(n\varphi(\gamma^*)-O(\ln n)\Bigr).
  \end{equation}
  \paragraph{Bounding Restricted Partition Functions.} We now assume \[
  p\ge P = \max\{P^*,P_0,\dots,P_{Q-1},\hat{P}_0,\dots,\hat{P}_{L-1}\}\]
  and upper bound $Z_i/Z^*$ and $\hat{Z}_i/Z^*$ for each $i$. To that end, fix $0\le i\le Q-1$ and observe that
  \begin{align}
     \varphi(\tau_i)-\varphi(\gamma^*) &=(1-\tau_i^2)\ln 2 + \beta\tau_i \sqrt{2\ln 2} - (1-(\gamma^*)^2)\ln 2 - \beta\gamma^*\sqrt{2\ln 2}\nonumber \\
     &=(\gamma^*-\tau_i)\bigl(\ln 2(\gamma^*+\tau_i) - \beta\sqrt{2\ln 2}\bigr)\nonumber \\
     &=-\ln 2(\gamma^*-\tau_i)^2\label{eq:lasttt},
  \end{align}
where~\eqref{eq:lasttt} uses the fact $\gamma^* = \beta/\sqrt{2\ln 2}$. Now,
  \begin{align}
      \frac{Z_i}{Z^*}&\le \exp\Bigl(n\bigl(\varphi(\tau_i)-\varphi(\gamma^*)\bigr)+n\beta\kappa^4\sqrt{2\ln 2}+O(\ln n)\Bigr)\nonumber \\
      & = \exp\Bigl(-n\ln 2(\gamma^*-\tau_i)^2+n\beta\kappa^4\sqrt{2\ln 2}+O(\ln n)\Bigr) \label{eq:use-lastttt}\\
      &\le \exp\Bigl(-n(\ln 2)\kappa^2 + n\beta\kappa^4\sqrt{2\ln 2}+O(\ln n)\Bigr)\label{eq:penultimate}
  \end{align}
  where~\eqref{eq:use-lastttt} uses~\eqref{eq:lasttt} and~\eqref{eq:penultimate} is obtained by noticing that $|\gamma^*-\tau_i|\ge \kappa$. Moreover, the upper bound~\eqref{eq:penultimate}  remains true also for $\hat{Z}_i/Z^*$. Additionally for $Z_n$, 
  \[
  \frac{Z_n}{Z^*}\le \exp\bigl(n(\varphi(0)-\varphi(\gamma^*)+O(\ln n)\bigr)\le \exp\Bigl(-n(\ln 2)\kappa^2 + O(\ln n)\Bigr) = e^{-\Theta(n)},
  \]
  for any $\kappa>0$. We now set
  \begin{equation}\label{eq:kappa-star}      
  \kappa^* = \left(\frac{\ln 2}{2\beta^2}\right)^{\frac14}.
  \end{equation}
  Then, provided $\kappa<\kappa^*$, we immediately obtain 
  \[
\exp\Bigl(-n(\ln 2)\kappa^2 + n\beta\kappa^4\sqrt{2\ln 2}+O(\ln n)\Bigr) = e^{-\Theta(n)}.
  \]
  Note that for $\beta<\sqrt{2\ln 2}$, $\kappa^*$ in~\eqref{eq:kappa-star} is at least $1/\sqrt{2}$. So, $\kappa<\bar{\kappa}=1/\sqrt{2}$ immediately ensures $\kappa<\kappa^*$. Lastly, as $\kappa=O_n(1)$, we have that $Q,L=O_n(1)$, and consequently, 
  \[
  \frac{1}{Z^*}\left(Z_n + \sum_{0\le i\le Q-1}Z_i + \sum_{0\le i\le L-1}\hat{Z}_i\right)\le e^{-\Theta(n)}.
  \]
  This establishes Proposition~\ref{prop:pf-dominated}, since
  \[
  Z= Z^* + Z_n+\sum_{0\le i\le Q-1}Z_i + \sum_{0\le i\le L-1}\hat{Z}_i. \qedhere
  \]
\end{proof}

\subsection{Proof of Theorem~\ref{thm:m-ogp}}\label{sec:pf-m-ogp}
Our proof is based on the \emph{first moment method}. Fix an $m\in\mathbb{N}$ and a $\gamma>1/\sqrt{m}$. For $0<\eta<\xi<1$ and $c>0$ to be tuned, fix an $\mathcal{I}\subset [0,\frac{\pi}{2}]$ with  $|\mathcal{I}|\subset 2^{cn}$. For $\tau\in\mathcal{I}$, recall the notation $\hat{J}^{(t)}_{i_1,\dots,i_p}(\tau)$ from Definition~\ref{def:admit}. Set
\[
\boldsymbol{\widehat{J}}^{(t)}(\tau) = \left(\hat{J}_{i_1,\dots,i_p}^{(t)}(\tau):1\le i_1,\dots,i_p\le n\right)\in(\R^n)^{\otimes p}.
\]
and let
\begin{equation}\label{eq:hamilton-interpol}
    H\left(\bs^{(t)},\boldsymbol{\hat{J}}^{(t)}(\tau)\right) = n^{-\frac{p+1}{2}}\sum_{1\le i_1,\dots,i_p\le n} \hat{J}^{(t)}_{i_1,\dots,i_p}(\tau)\bs^{(t)}_{i_1}\cdots \bs^{(t)}_{i_p}.
\end{equation}
Next, for 
\begin{equation}\label{eq:overlap-set}
\mathcal{F}(m,\xi,\eta)\triangleq \Bigl\{\left(\bs^{(1)},\dots,\bs^{(m)}\right):\xi-\eta\le n^{-1}\ip{\bs^{(i)}}{\bs^{(j)}}\le \xi,1\le i<j\le m\Bigr\}
\end{equation}
set
\begin{equation}\label{eq:rv-N}
    M=\bigl|S(\gamma,m,\xi,\eta,\mathcal{I})\bigr|=\sum_{(\bs^{(1)},\dots,\bs^{(m)})\in \mathcal{F}(m,\xi,\eta)} \ind\left\{\exists\tau_1,\dots,\tau_m\in\mathcal{I}:\min_{1\le i\le m}H\left(\bs^{(i)},\boldsymbol{\hat{J}}^{(i)}(\tau_i)\right)\ge \gamma\sqrt{2\ln 2}\right\}.
\end{equation}
We will establish $\mathbb{E}[M]=\exp(-\Theta(n))$ for suitable $0<\eta<\xi<1$; which will then yield the result via Markov's inequality:
\begin{equation}\label{eq:markovvvv}
\mathbb{P}\Bigl[S(\gamma,m,\xi,\eta,\mathcal{I})\ne\varnothing\Bigr] =\mathbb{P}[M\ge 1]\le \mathbb{E}[M] = \exp(-\Theta(n)).
\end{equation}
We next estimate $\mathbb{E}[M]$ by upper bounding the cardinality, $|\mathcal{F}(m,\xi,\eta)|$, and the probability term. 
\paragraph{Counting term.} Fix $m\in\mathbb{N}$ and $0<\eta<\xi<1$. We claim
\begin{lemma}\label{lemma:counting}
\[
\bigl|\mathcal{F}(m,\xi,\eta)\bigr|\le \exp_2\left(n+n(m-1)h\left(\frac{1-\xi+\eta}{2}\right)+O(\log_2 n)\right).
\]
\end{lemma}
\begin{proof}[Proof of Lemma~\ref{lemma:counting}]
Note that for any $\bs,\bs'\in\bincube$,
\[
\ip{\bs}{\bs'}  = n-2d_H(\bs,\bs').
\]
There are $2^n$ choices for $\bs^{(1)}$. Having fixed a $\bs^{(1)}$; any $\bs^{(i)}$, $2\le i\le m$, can be chosen in 
\[
\displaystyle\sum_{\substack{\rho:\rho n\in\mathbb{N}\\ \frac{1-\xi}{2}\le \rho\le \frac{1-\xi+\eta}{2}}} \binom{n}{\rho n}\le \binom{n}{\frac{1-\xi+\eta}{2} n}n^{O(1)} 
\]
different ways, subject to $\xi-\eta\le n^{-1}\ip{\bs^{(1)}}{\bs^{(i)}}\le \xi$. Next, for any $\rho\in(0,1)$, $\binom{n}{n\rho} = \exp_2\bigl(nh(\rho) + O(\log_2 n)\bigr)$ by Stirling's approximation. Combining these, and the fact $m=O(1)$ (as $n\to\infty$), we establish Lemma~\ref{lemma:counting}.
\end{proof}
\paragraph{Probability term} We next upper bound the probability term. To that end, we first establish that it suffices to consider $\tau_1=\cdots=\tau_m=0$. To that end, we recall Slepian's lemma~\cite{slepian1962one}:
\begin{lemma}\label{lemma:Slepian}
    Let $X=(X_1,\dots,X_n)$ and $Y=(Y_1,\dots,Y_n)$ be multivariate normal random vectors such that $\mathbb{E}[X_i]=\mathbb{E}[Y_i]=0,\forall i$, $\mathbb{E}[X_i^2]=\mathbb{E}[Y_i^2],\forall i$, and
    \begin{align*}
        \mathbb{E}[X_iX_j]\le \mathbb{E}[Y_iY_j],\quad \text{for}\quad 1\le i<j\le n.
    \end{align*}
    Fix any $c_1,\dots,c_n\in\R$. Then,
    \[
    \mathbb{P}\bigl[X_i\le c_i,\forall i\bigr] \le  \mathbb{P}\bigl[Y_i\le c_i,\forall i\bigr].
    \]
    \end{lemma}
In particular, applying Lemma~\ref{lemma:Slepian} to $-X=(-X_1,\dots,-X_n)$ and $-Y=(-Y_1,\dots,-Y_n)$, we immediately obtain \[
    \mathbb{P}\bigl[X_i\ge c_i,\forall i\bigr] \le  \mathbb{P}\bigl[Y_i\ge c_i,\forall i\bigr],
    \]
    for all $c_1,\dots,c_n\in\R$. Now, fix a $\left(\bs^{(1)},\dots,\bs^{(m)}\right)\in\mathcal{F}(m,\xi,\eta)$, $\tau_1,\dots,\tau_m\in\mathcal{I}$, and denote by $\Sigma(\bs^{(1)},\dots,\bs^{(m)};\tau_1,\dots,\tau_m)\in\R^{m\times m}$ the covariance matrix associated to 
    \[
    \left(\sqrt{n}H\bigl(\bs^{(i)},\boldsymbol{\widehat{J}}^{(i)}(\tau_i)\bigr):1\le i\le m\right)\in\R^m.
    \]
    We first verify that 
     \begin{equation}\label{eq:TO-VERIFY}
\Bigl(\Sigma\bigl(\bs^{(1)},\dots,\bs^{(m)};\tau_1,\dots,\tau_m\bigr)\Bigr)_{k,\ell} \le \Bigl(\Sigma\bigl(\bs^{(1)},\dots,\bs^{(m)};0,\dots,0\bigr)\Bigr)_{k,\ell},\quad \forall 1\le k<\ell \le m.
         \end{equation}
    Observe that
    \begin{align}
        &\mathbb{E}\left[\sqrt{n}H\left(\bs^{(k)}\hat{\boldsymbol{J}}^{(k)}(\tau_k)\right) \cdot \sqrt{n}H\left(\bs^{(\ell)}\hat{\boldsymbol{J}}^{(\ell)}(\tau_\ell)\right)\right]\nonumber \\
        &=\frac{1}{n^p}\sum_{\substack{1\le i_1,\dots,i_p\le n\\ 1\le i_1',\dots,i_p'\le n}}\mathbb{E}\left[\left(\cos(\tau_k)J_{i_1,\dots,i_p}^{(0)} + \sin(\tau_k)J_{i_1,\dots,i_p}^{(k)}\right)\left(\cos(\tau_\ell)J_{i_1',\dots,i_p'}^{(0)} + \sin(\tau_\ell)J_{i_1',\dots,i_p'}^{(\ell)}\right)\right] \left(\prod_{1\le j\le p}\bs^{(k)}_{i_j}\bs^{(\ell)}_{i_j'}\right)\nonumber\\
        &=\frac{1}{n^p}\sum_{1\le i_1,\dots,i_p\le n}\cos(\tau_k)\cos(\tau_\ell)\left(\prod_{1\le j\le p}\bs^{(k)}_{i_j}\bs^{(\ell)}_{i_j}\right) \nonumber\\
        &=\cos(\tau_k)\cos(\tau_\ell) \left(\frac{\ip{\bs^{(k)}}{\bs^{(\ell)}}}{n}\right)^p\label{eq:COV-EQ}.
    \end{align}
    Using now the fact $0\le \cos(\tau_k),\cos(\tau_\ell)\le 1$ and the fact $\ip{\bs^{(k)}}{\bs^{(\ell)}}\ge 0$ per~\eqref{eq:overlap-set}, we verify~\eqref{eq:TO-VERIFY}.
   
    Applying Lemma~\ref{lemma:Slepian}, we thus obtain 
    \begin{equation}\label{eq:suffices-to-consider-tau-0}
        \sup_{\tau_1,\dots,\tau_i\in[0,\frac{\pi}{2}]}\mathbb{P}\left[H\left(\bs^{(i)},\boldsymbol{\widehat{J}}^{(i)}(\tau_i)\right)\ge \gamma \sqrt{2\ln 2},\forall i\right] \le \mathbb{P}\left[H\left(\bs^{(i)},\boldsymbol{\widehat{J}}^{(i)}(0)\right)\ge \gamma \sqrt{2\ln 2},\forall i\right].
    \end{equation}
   Setting $H(\bs^{(i)})\triangleq H\left(\bs^{(i)},\boldsymbol{\widehat{J}}^{(i)}(0)\right), 1\le i\le m$ for convenience, it thus suffices to control 
   \begin{equation}\label{eq:control-thissss}
    \mathbb{P}\left[H\bigl(\bs^{(i)}\bigr)\ge \gamma \sqrt{2\ln 2},\forall i\right].
   \end{equation}
We next record several useful facts in the following Lemma.
\begin{lemma}\label{lemma:prob}
Fix any $(\bs^{(i)}:1\le i\le m)\in\mathcal{F}(m,\xi,\eta)$ and let
\[
n^{-1}\ip{\bs^{(k)}}{\bs^{(\ell)}} = \xi-\eta_{k\ell},\quad 1\le k<\ell\le m.
\]  
Denote by $\Sigma(\boldsymbol{\eta})$ the covariance matrix of $\left(\sqrt{n}H(\bs^{(i)}):1\le i\le m\right)\in\R^m$, where $\boldsymbol{\eta}=(\eta_{k\ell}:1\le k<\ell\le m)\in\R^{m(m-1)/2}$. 
Then, the following holds.
\begin{itemize}
    \item[(a)] For every $1\le k<\ell\le m$, $\eta_{k\ell}\ge 0$; $\|\boldsymbol{\eta}\|_\infty\le \eta$, $\bigl(\Sigma(\boldsymbol{\eta})\bigr)_{ii}=1$, $1\le i
\le m$, and
\[
\bigl(\Sigma(\boldsymbol{\eta})\bigr)_{k\ell}=\bigl(\Sigma(\boldsymbol{\eta})\bigr)_{\ell k}=(\xi-\eta_{k\ell})^p,\quad 1\le k<\ell\le m.
\]
    \item[(b)]  We have $\Sigma(\boldsymbol{\eta})= \Sigma+E$, where 
    \begin{equation}\label{eq:sigma}
\Sigma = (1-\xi^p)I_{m\times m}+\xi^p\boldsymbol{1}\boldsymbol{1}^T 
\end{equation}
$E_{ii}=0$ for $1\le i\le m$ and $0\le E_{k\ell}\le p\eta \xi^{p-1}$
    for $1\le k<\ell\le m$. Consequently, 
    $\Sigma(\boldsymbol{\eta})$ is positive definite (PD) provided 
    \[
    \eta < \frac{1-\xi^p}{mp\xi^{p-1}}.
    \]
    \item[(c)] For $\Sigma$ in~\eqref{eq:sigma}, 
    \[
    |\Sigma| = (1-\xi^p)^{m-1}\left(1+(m-1)\xi^p\right)
    \]
    and
    \[
    \Sigma^{-1} = \frac{1}{1-\xi^p} I - \frac{\xi^p}{(1-\xi^p)\left(1+(m-1)\xi^p\right)}\boldsymbol{1}\boldsymbol{1}^T.
    \]
    \item[(d)] Provided that $\eta$ is small enough, we have
    \begin{align*}
    &\mathbb{P}\left[\min_{1\le i\le m}H(\bs^{(i)})\ge \gamma\sqrt{2\ln 2}\right] \\
    &\le n^{\frac{m}{2}} \left(\gamma\sqrt{\frac{\ln 2}{\pi}}\right)^m \left(\prod_{i\le m}\ip{e_i}{\Sigma(\boldsymbol{\eta})^{-1}\boldsymbol{1}}\right)\bigl|\Sigma(\boldsymbol{\eta})\bigr|^{-\frac12}\exp_2\Bigl(-\gamma^2 n \cdot \boldsymbol{1}^T\Sigma(\boldsymbol{\eta})^{-1}\boldsymbol{1}\Bigr),
   \end{align*}
   where $e_i\in\R^m$ is the $i{\rm th}$ unit vector.
\end{itemize}
\end{lemma}
\begin{proof}[Proof of Lemma~\ref{lemma:prob}]
\begin{itemize}
    \item[(a)] As $(\bs^{(i)}:1\le i\le m)\in\mathcal{F}(m,\xi,\eta)$, we have $\xi-\eta\le n^{-1}\ip{\bs^{(i)}}{\bs^{(j)}}\le \xi$, yielding $\eta_{ij}\ge 0$ and $\|\boldsymbol{\eta}\|_\infty\le \eta$. Clearly $\sqrt{n}H(\bs^{(i)})\sim \cN(0,1)$ so that $\Sigma(\boldsymbol{\eta})_{ii}=1,\forall i$.  The expression for $(\Sigma(\boldsymbol{\eta}))_{k\ell}$ follows immediately from~\eqref{eq:COV-EQ} by taking $\tau_k=\tau_\ell=0$.
    \item[(b)] The expression $\Sigma(\boldsymbol{\eta})=\Sigma+E$ for $\Sigma$ in~\eqref{eq:sigma} is clear. We next observe that for any $x\in[\xi-\eta,\xi]$, 
\begin{equation}\label{eq:diff}
\xi^p - x^p = (\xi-x)\sum_{0\le i\le p-1}\xi^{p-1-i}x^i\le p\eta \xi^{p-1}.
\end{equation}
Using~\eqref{eq:diff}, $E_{ii}=0$ and $0\le E_{k\ell}\le p\eta\xi^{p-1}$ for $1\le k<\ell\le m$. In particular, $\|E\|_2\le \|E\|_F\le mp\eta \xi^{p-1}$. Noting that the smallest eigenvalue of $\Sigma$ is $1-\xi^p$, the result follows from Theorem~\ref{thm:hw}.
\item[(c)] Noting that the eigenvalues of $\boldsymbol{1}\boldsymbol{1}^T$ are $m$ with multiplicity 1 and $0$ with multiplicity $m-1$, the expression for $|\Sigma|$ follows. 

For $\Sigma^{-1}$, we apply Sherman-Morrison formula, Theorem~\ref{thm:sm}, with
\[
A=I_{m\times m},\quad u=v=\sqrt{\frac{\xi^p}{1-\xi^p}}\boldsymbol{1}\in\R^{m\times 1}.
\]
Check that 
\[
1+v^T A^{-1}u = 1 + v^Tv= 1+\frac{m\xi^p}{1-\xi^p}= \frac{1+(m-1)\xi^p}{1-\xi^p}\ne 0.
\]
as $\boldsymbol{1}^T\boldsymbol{1} = m$. Finally, 
\begin{align*}
    \Sigma^{-1} &= \frac{1}{1-\xi^p}\left(I-\frac{\xi^p}{1-\xi^p}\boldsymbol{1}\boldsymbol{1}^T\right)^{-1} \\
    &=\frac{1}{1-\xi^p}I -\frac{1}{1-\xi^p} \frac{\frac{\xi^p}{1-\xi^p} \boldsymbol{1}\boldsymbol{1}^T}{1+\frac{m\xi^p}{1-\xi^p}} \\
    &=\frac{I}{1-\xi^p} - \frac{\xi^p}{(1-\xi^p)\left(1+(m-1)\xi^p\right)}\boldsymbol{1}\boldsymbol{1}^T.
\end{align*}
\item[(d)] We apply Theorem~\ref{thm:multiv-tail} with $\boldsymbol{t} = (\gamma\sqrt{2n\cdot \ln 2})\boldsymbol{1}\in\R^m$. Before doing so, we have to ensure $\bigl(\Sigma(\boldsymbol{\eta})^{-1}\boldsymbol{t}\bigr)_i>0$ for $1\le i\le m$. Note that for any $i\le m$, 
the map $\boldsymbol{\eta}\to(\Sigma(\boldsymbol{\eta})^{-1}\boldsymbol{t})_i$ is continuous and $\boldsymbol{\eta}\in[0,\eta]^{m(m-1)/2}$ belongs to a compact domain. Thus, it suffices to verify that entrywise
\[
\Sigma^{-1}\boldsymbol{t}>0\iff \Sigma^{-1}\boldsymbol{1}>0.
\]
We recall $\Sigma^{-1}$ from part ${\rm (b)}$. Noting that $\boldsymbol{1}^T\boldsymbol{1}=m$, we get
\begin{align*}
    \Sigma^{-1}\boldsymbol{1} &= \left(\frac{1}{1-\xi^p}-\frac{m\xi^p}{(1-\xi^p)\left(1+(m-1)\xi^p\right)}\right)\boldsymbol{1}\\
    &=\frac{1}{1-\xi^p+m\xi^p}\boldsymbol{1},
\end{align*}
which is clearly entrywise positive. Finally, since
\[
\exp\left(-\frac{\boldsymbol{t}^T\Sigma(\boldsymbol{\eta})^{-1}\boldsymbol{t}}{2}\right) = \exp\left(-\frac{\gamma^2 2n \ln 2}{2} \boldsymbol{1}^T\Sigma(\boldsymbol{\eta})^{-1}\boldsymbol{1}\right) = \exp_2\left(-\gamma^2\cdot n\cdot \boldsymbol{1}^T\Sigma(\boldsymbol{\eta})^{-1}\boldsymbol{1}\right)
\]
we conclude the proof after some algebraic manipulations.
\end{itemize}
\end{proof}
\paragraph{Upper Bounding Probability Terms.} Let the eigenvalues of the PD matrix $\Sigma(\boldsymbol{\eta})$ be $0<\lambda_1\le\cdots\le\lambda_m$. Using Theorem~\ref{thm:hw}, Lemma~\ref{lemma:prob}${\rm (b)}$, and the fact that $\eta<\xi$, we have
\[
\bigl|\lambda_1-(1-\xi^p)\bigr|\le \|E\|_F\le mp\eta\xi^{p-1}\le mp\xi^{p}\implies \lambda_1\ge 1-2mp\xi^p.
\]
The eigenvalues of $\Sigma(\boldsymbol{\eta})^{-1}$ are $\lambda_1^{-1},\dots,\lambda_m^{-1}$. Consequently,
\begin{equation}\label{eq:det}
\bigl|\Sigma(\boldsymbol{\eta})\bigr|^{-\frac12} = \prod_{1\le i\le m}\lambda_i^{-\frac12}\le \bigl(1-2mp\xi^p\bigr)^{-\frac{m}{2}} = O_n(1).
\end{equation}
Next, using the Cauchy-Schwarz inequality and the fact $\|\boldsymbol{1}\|_2=\sqrt{m}$, we obtain
\[
\ip{e_i}{\Sigma(\boldsymbol{\eta})^{-1}\boldsymbol{1}} \le \|e_i\|_2 \cdot \|\Sigma(\boldsymbol{\eta})^{-1}\boldsymbol{1}\|_2\le \frac{\sqrt{m}}{1-2mp\xi^p}. 
\]
Hence
\begin{equation}\label{eq:prod-coordinate}
    \prod_{i\le m}\ip{e_i}{\Sigma(\boldsymbol{\eta})^{-1}\boldsymbol{1}}\le m^{\frac{m}{2}}\bigl(1-2mp\xi^p\bigr)^{-m}=O_n(1).
\end{equation}
We finally control $\boldsymbol{1}^T\Sigma(\boldsymbol{\eta})^{-1}\boldsymbol{1}$. Using Wielandt-Hoffman inequality and Lemma~\ref{lemma:prob}${\rm (b)}$,  
\[
\Bigl|\lambda_m - \bigl(1+(m-1)\xi^p\bigr)\Bigr|\le \|E\|_F \le mp\eta\xi^{p-1}\implies \lambda_m\le 1+2mp\xi^{p}.
\]
Diagonalize now $\Sigma(\boldsymbol{\eta})$ as 
\[
\Sigma(\boldsymbol{\eta})^{-1}=Q(\boldsymbol{\eta})^T\Lambda Q(\boldsymbol{\eta}),
\]
where
\[
Q(\boldsymbol{\eta})\in\R^{m\times m}\quad\text{with}\quad Q(\boldsymbol{\eta})^T Q(\boldsymbol{\eta}) = Q(\boldsymbol{\eta})Q(\boldsymbol{\eta})^T =I_m
\]and \[
\Lambda={\rm diag}\bigl(\lambda_i^{-1}:1\le i\le m\bigr)\in\R^{m\times m}
\]
is a diagonal matrix. Letting $Q(\boldsymbol{\eta})\boldsymbol{1}\triangleq\bigl(\zeta_i(\boldsymbol{\eta}):i\le m\bigr)$, we obtain
\begin{equation}\label{eq:integrand}
    \boldsymbol{1}^T \Sigma(\boldsymbol{\eta})^{-1}\boldsymbol{1} = \sum_{1\le i\le m}\lambda_i^{-1}\zeta_i(\boldsymbol{\eta})^2\ge \frac{1}{1+2mp\xi^{p}}\|Q(\boldsymbol{\eta})\boldsymbol{1}\|_2^2=\frac{m}{1+2mp\xi^{p}}.
\end{equation}
We now have all necessary ingredients for controlling $\mathbb{E}[M]$.
\paragraph{Estimating the Expectation.} Assume $0<\eta<\xi<1$ and $m\in\mathbb{N}$, all of which are fixed as $n\to\infty$. (We will tune $\eta$ eventually.) We have 
\begin{align}
\mathbb{E}[M] &\le \exp_2\left(n+n(m-1)h\left(\frac{1-\xi+\eta}{2}\right)+O(\log_2 n)\right) \times n^{\frac{m}{2}} \left(\gamma\sqrt{\frac{\ln 2}{\pi}}\right)^m\times 2^{cmn}\nonumber\\
& \sup_{\boldsymbol{\eta}:\|\boldsymbol{\eta}\|_\infty\le \eta}\left(\bigl|\Sigma(\boldsymbol{\eta})\bigr|^{-\frac12}\prod_{i\le m}\ip{e_i}{\Sigma(\boldsymbol{\eta})^{-1}\boldsymbol{1}}\right)\cdot \sup_{\boldsymbol{\eta}:\|\boldsymbol{\eta}\|_\infty\le \eta}\exp_2\Bigl(-\gamma^2 n \cdot \boldsymbol{1}^T\Sigma(\boldsymbol{\eta})^{-1}\boldsymbol{1}\Bigr)\label{eq:up-bd1}\\
&\le \exp_2\left(n+nmh\left(\frac{1-\xi+\eta}{2}\right)-n\frac{m\gamma^2}{1+2mp\xi^p}+cmn+O(\log_2 n)\right)\label{eq:up-bd2}.
\end{align}
Here,~\eqref{eq:up-bd1} follows by combining the counting bound per Lemma~\ref{lemma:counting}, probability term arising from Lemma~\ref{lemma:prob}${\rm (d)}$,~\eqref{eq:suffices-to-consider-tau-0}, as well as a union bound over all $\tau_1,\dots,\tau_i\in\mathcal{I}$ (recall $|\mathcal{I}|\le 2^{cn}$). Next,~\eqref{eq:up-bd2} follows by upper bounding~\eqref{eq:up-bd1} further via~\eqref{eq:det},~\eqref{eq:prod-coordinate} and~\eqref{eq:integrand}. Hence,
\begin{align*}
    \mathbb{E}[M]\le \exp_2\Bigl(n\Psi(m,\xi,\eta,p,c)+O(\log_2 n)\Bigr),
\end{align*}
where 
\begin{equation}\label{eq:free-energy}
    \Psi(m,\xi,\eta,p,c) = 1+mh\left(\frac{1-\xi+\eta}{2}\right) - \frac{m\gamma^2}{1+2mp\xi^p}+cm.
\end{equation}
\paragraph{Making $\Psi$ Negative.}
Recall that $\gamma>1/\sqrt{m}$, hence
$\delta\triangleq m\gamma^2-1>0$. Choosing $\xi$ sufficiently close to 1 and $\eta$ small enough (while retaining $\eta<\xi$), we first ensure
\[
mh\left(\frac{1-\xi+\eta}{2}\right) \le \frac{\delta}{4}.
\]
Note that for fixed $m\in\mathbb{N}$ and $\xi<1$,  $1+2mp\xi^{p}\to 1$ as $p\to\infty$. Choose $P^*$ such that
\[
\frac{m\gamma^2}{1+2mp\xi^{p-1}}\ge 1+\frac{\delta}{2}
\]
for all $p\ge P^*$. Lastly, choose $c\le \frac{\delta}{8m}$. With $0<\eta<\xi<1$ as above and $p\ge P^*$, we thus have
\[
\Psi(m,\xi,\eta,p)\le -\frac{\delta}{8}<0,
\]
from which $\mathbb{P}[M\ge 1]\le \exp\bigl(-\Theta(n)\bigr)$ per~\eqref{eq:markovvvv}.
\subsubsection*{Acknowledgments}
The first author is supported in part by NSF grant DMS-2015517.  The second author acknowledges the support of the Natural Sciences and Engineering Research Council of Canada (NSERC) and the Canada Research Chairs programme. La recherche du deuxi\`eme auteur a \'et\'e enterprise gr\^ace, en partie, au 
soutien financier du Conseil de Recherches en Sciences Naturelles et en G\'enie du Canada (CRSNG),  [RGPIN-2020-04597, DGECR-2020-00199], et du Programme des chaires de recherche du Canada. The third author is supported by a Columbia University, Distinguished Postdoctoral Fellowship in Statistics.
\bibliographystyle{amsalpha}
\bibliography{bibliography}

\end{document}